\documentclass{amsart}
\usepackage{latexsym,amsfonts,amssymb,amsthm,mathrsfs,
  amsmath,amscd,enumerate,amscd,color, enumitem}
\usepackage[utf8x]{inputenc}

 \newtheorem{thm}{Theorem}[section]
 \newtheorem{cor}[thm]{Corollary}
 \newtheorem{lem}[thm]{Lemma}
 
\theoremstyle{definition}
\newtheorem{defn}[thm]{Definition}
 \theoremstyle{remark}
 \newtheorem{rem}[thm]{Remark}

\newcommand{\mbb}[1]{\ensuremath{\mathbb{#1}}}
\newcommand{\mc}[1]{\ensuremath{\mathcal{#1}}}

\newcommand{\R}{\mathbb{R}} 

\newcommand{\pp}{{p(\cdot)}}
\newcommand{\cpp}{{p'(\cdot)}}
\newcommand{\Lp}{L^{p(\cdot)}}

\newcommand{\Pp}{\mathcal P}
\newcommand{\qq}{{q(\cdot)}}

\numberwithin{equation}{section}
\allowdisplaybreaks

\begin{document}

\title[The Calderón operator and the Stieltjes transform on
$L^{p(\cdot)}_w$] {The Calderón operator and the Stieltjes transform
  on variable Lebesgue spaces with weights}

\author[D. Cruz-Uribe]{David Cruz-Uribe}
\address{David Cruz-Uribe, OFS \newline 
	Department of Mathematics, University of Alabama, Tuscaloosa, AL 35487, USA}
\email{dcruzuribe@ua.edu}

\author[E. Dalmasso]{Estefan\'ia Dalmasso}
\address{Estefan\'ia Dalmasso\newline
	Instituto de Matem\'atica Aplicada del Litoral, UNL, CONICET, FCE/FIQ.\newline Colectora Ruta Nac. N° 168, Paraje El Pozo, S3007ABA, Santa Fe, Argentina}
\email{edalmasso@santafe-conicet.gov.ar}

\author[F. J. Mart\'in-Reyes]{Francisco J. Mart\'in-Reyes}
\address{Francisco J. Mart\'in-Reyes, Pedro Ortega Salvador\newline
	Facultad de Ciencias, Universidad de M\'alaga\newline
	Campus de Teatinos, 29071 M\'alaga, Spain}
\email{martin\_reyes@uma.es, portega@uma.es}

\author[P. Ortega]{Pedro Ortega Salvador}
\address{}
\email{}

\thanks{The first author was supported by research funds provided by
  the Dean of the College of Arts \& Sciences, the University of
  Alabama.  The second author was supported by CONICET, ANPCyT and
  CAI+D (UNL). The third and fourth authors were supported by
  grants MTM2011-28149-C02-02 and MTM2015-66157-C2-2-P (MINECO/FEDER)
  of the Ministerio de Economía y Competitividad (MINECO, Spain) and
  grant FQM-354 of the Junta de Andalucía, Spain.  The initial stages
  of this project were begun while the first two authors were visiting
  Málaga, and they want to thank the third and fourth authors for
  their hospitality.  }





\begin{abstract} 
  We characterize the weights for the Stieltjes transform and the
  Calder\'on operator to be bounded on the weighted variable Lebesgue spaces
  $L_w^{p(\cdot)}(0,\infty)$, assuming that the exponent function
  $\pp$ is log-H\"older continuous at the origin and at infinity. We obtain a single 
  Muckenhoupt-type condition by means of a maximal operator defined
  with respect to the basis of intervals $\{ (0,b) : b>0\}$ on
  $(0,\infty)$. Our results extend those in \cite{DMRO1} for the
  constant exponent $L^p$ spaces with
  weights. We also give two applications: the first  is a weighted
  version of Hilbert's inequality on variable Lebesgue spaces, and the
  second  generalizes the results in \cite{SW}
  for integral operators to the variable exponent setting.
\end{abstract}

\maketitle

\section{Introduction and results}

In this paper we consider two classical operators: the generalized
Stieltjes transform $S_\lambda$ and the generalized Calder\'on
operator $\mc C_\lambda$, where $0<\lambda \leq 1$, defined for
non-negative functions $f$ on $(0,\infty )$ by
$$
S_\lambda f(x)=\int_0^{\infty}\frac{f(y)}{(x+y)^\lambda}\,dy
$$
and
\[\mc C_\lambda f(x)=\frac{1}{x^\lambda}\int_0^x f(y)\,dy+\int_x^\infty
  \frac{f(y)}{y^\lambda}\,dy.  \]

The Calderón operator $\mc C=\mc C_1$ plays an important role in the
theory of interpolation:  see \cite{BMR}.   More generally, we have
that for $\lambda>0$, $\mc C_\lambda=H_\lambda +H_\lambda^*$, where
\[ H_\lambda f(x) = \frac{1}{x^\lambda}\int_0^x f(y)\,dy \]
is a Hardy-type operator and $H_\lambda^*$ its adjoint. 
The Stieltjes transform
$S=S_1$ is, formally, the same as $\mc{L}^2=\mc{L}\circ \mc{L}$, where
$\mc{L}$ is the Laplace transform. A classical reference for
the Stieltjes transform is the monograph by D. Widder \cite{W}.

These two operators clearly satisfy
$2^{-\lambda}\mc C_\lambda f(x)\le S_\lambda f(x)\le \mc C_\lambda
f(x)$, so $S_\lambda$ is bounded on a Banach function space if and
only if $\mc C_\lambda$ is.  Hereafter, given functions $f,g\geq 0$ we will write $f\lesssim g$ if
there exists $c>0$ such that $f\leq cg$.
If $f\lesssim g$ and $g\lesssim f$ hold, we will write $f\sim g$.
Thus we have that $S_\lambda f \sim \mc{C}_\lambda f$.

We shall also consider the operator
\[S^\alpha f(x)=\int_0^\infty \frac{|x-t|^\alpha}{(x+t)^{\alpha+1}} f(t) dt,\quad \alpha\geq 0, \]
and
$\mc C^\alpha$, which is the sum of the Riemann-Liouville and Weyl
averaging operators: 
\[\mc C^\alpha f(t)
=
I^\alpha f(t)+J^\alpha  f(t)
=
\frac{\alpha+1}{t^{\alpha+1}}\int_0^t (t-x)^\alpha f(x)\,dx
 +
(\alpha+1)\int_t^{\infty} \frac{(x-t)^\alpha}{x^{\alpha+1}} f(x)\,dx.\] 
It is clear that if $\alpha=0$, then $S^\alpha$, $\mc C^\alpha$,
$I^\alpha$ and $J^\alpha$ are $S$, $\mc C$, $H_1$ and $H_1^*$,
respectively. Moreover, $I^\alpha f\lesssim H_1 f$,
$J^\alpha f\lesssim H_1^* f$, $\mc C^\alpha f\lesssim \mc C_1 f$ and
$S^\alpha f \sim \mc{C}^\alpha f$ for non-negative measurable
functions $f$.

\medskip

To put our results into context, we briefly review the history of
weighted norm inequalities for the Calder\'on operator $\mc C_\lambda$ and the
Stieltjes transform $S_\lambda$, which in turn depend on the weighted norm
inequalities for the Hardy operator $H_\lambda$.   Muckenhoupt~\cite{M1}
established two-weight norm inequalities for the Hardy operator; this
implicitly gave bounds for the Calder\'on operator using this
condition and its dual.   A different condition for the Stieltjes transform, expressed in terms of
the operator $S_\lambda$ applied to the pair of weights, was discovered by
Andersen \cite{A}.   As a consequence, he proved the following
one-weight condition.

\begin{thm} \label{thm:andersen}
Given $0<\lambda \leq 1$ and $1<p  <\frac{1}{1-\lambda}$, define $q\geq p$ by
$\frac{1}{q}=\frac{1}{p}-(1-\lambda)$.  Then 
$S_\lambda : L^p(w^p) \rightarrow L^q(w^q)$ if and only if the weight
$w$ satisfies the $A_{p,q,0}$ condition:
\begin{equation}\label{Apq0}
\sup_{b>0}\left(\frac 1b\int_0^b w^q\,dx\right)^{\frac 1p}\left(\frac
  1b\int_0^bw^{-p'}\,dx\right)^{\frac 1{p'}}<\infty,
\end{equation}
where $p'$ stands for the H\"older conjugate exponent of $p$. 
\end{thm}

The $A_{p,q,0}$ condition is a weaker version of the $A_{p,q}$ condition
introduced by Muckenhoupt and Wheeden~\cite{MW} to characterize the
weighted norm inequalities for fractional integrals and fractional
maximal operators.  (See also~\cite{CruzUribe:2016ji}.)  In the
one-weight case the
restriction on $p$ and $q$ is natural:  by homogeneity, if $S_\lambda
: L^p(0,\infty) \rightarrow L^q(0,\infty)$, then
$\frac{1}{q}=\frac{1}{p}-(1-\lambda)$.  

For other results on weighted norm inequalities for the Hardy
operator, the Calder\'on operator and the Stieltjes transform, we
refer the reader to  Sinnammon \cite{S} and  Gogatishvili, {\em et
  al.} \cite{GKP13, GKPW,GPSW}.

A different approach to the one-weight inequalities for $S_\lambda$
and $\mc C_\lambda$ in the case $\lambda=1$ was developed by
Duoandikoetxea, Mart\'{\i}n-Reyes and Ombrosi \cite{DMRO1}.  They
introduced a maximal operator $N$ defined with respect to the basis
$\mathcal{B} = \{ (0,b) : b>0 \}$: for $f\in L^1_{\rm loc}(0,\infty)$ and
$x\in (0,\infty)$,
\[Nf(x)=\sup\limits_{b>x}\frac{1}{b}\int_0^b |f(y)|\,dy.\]
They proved the following weighted norm inequality. 

\begin{thm} \label{thm:DMRO-thm}
Given $1<p  <\infty$, 
$N : L^p(w) \rightarrow L^p(w)$ if and only if the weight
$w$ satisfies the $A_{p,0}$ condition:
\begin{equation}\label{Ap0}
\sup_{b>0}\left(\frac 1b\int_0^b w\,dx\right)
\left(\frac
  1b\int_0^bw^{1-p'}\,dx\right)^{p-1}<\infty.
\end{equation}
\end{thm}

The $A_{p,0}$ condition is analogous to the Muckenhoupt $A_p$
condition, which characterizes weighted norm inequalities for the
Hardy-Littlewood maximal operator~\cite{M2} (see
also~\cite{CruzUribe:2016ji}).  This class is related to the
$A_{p,q,0}$ class given above:  if $q=p$ and $w\in A_{p,p,0}$, then
$w^p\in A_{p,0}$.  

For non-negative functions $f$, we have that $Nf\leq \mc
C f$:  given $0<x<b$
\[ \frac{1}{b}\int_0^b f(y)\,dy 
\leq 
\frac{1}{x}\int_0^x f(y)\,dy
+
\int_x^b \frac{f(y)}{y}\,dy 
\leq
S f(x); \]
if we take the supremum over all such $b$ we get the desired inequality. 
Similarly,  we also have that $H f\leq Nf$.   By a straightforward
duality argument using the Hardy operators, in~\cite{DMRO1} they
proved the following result.

\begin{thm} \label{thm:DMRO-calderon-thm}
Given $1<p  <\infty$, 
$\mc C : L^p(w) \rightarrow L^p(w)$ if and only if the weight
$w$ satisfies the $A_{p,0}$ condition; a similar result holds for
$S$.
\end{thm}

\medskip

In this paper, our goal is to generalize these results in two ways.
First, we extend the approach in~\cite{DMRO2} to give a new proof
of Theorem~\ref{thm:andersen} using the maximal operator $N$.  We will do so
using a Hedberg type inequality~\cite{Hed}.  More importantly, we 
extend all of these results to the scale of variable Lebesgue spaces. These are a generalization of the classical Lebesgue spaces, with the
constant exponent $p$ replaced by an exponent function $\pp$.  They
were introduced by Orlicz~\cite{O} in 1931; harmonic analysis
on these spaces has been studied intensively for the past 25 years.
We refer the reader to the monographs~\cite{CUFbook, DHHR} for
a comprehensive history.

To state our results we first  introduce some basic definitions; for
more information we refer the reader to the above books and also
to~\cite{KR}.  Let $\Pp(0,\infty)$ denote the collection of bounded
measurable functions $ \pp :(0,\infty)\rightarrow [1,\infty)$.  For a
measurable subset $E$ of $(0,\infty)$, let
\[p^-_E=\inf_{x\in E} p(x), \quad p^+_E=\sup_{x\in E} p(x);\]
for brevity we will simply write  $p^-=p^-_{(0,\infty)}$ and $p^+=p^+_{(0,\infty)}$.
Thus, we can write
\[\Pp(0,\infty)=\{\pp :(0,\infty)\rightarrow [1,\infty) \textrm{ with } p^+<\infty\}.\]

As in the constant exponent case, define the conjugate exponent
$\cpp$ pointwise by
\[\frac{1}{p(x)}+\frac{1}{p'(x)}=1\]
on every $x\in (0,\infty)$. Notice that, if $p(x)=1$, then
$p'(x)=\infty$ so $\cpp\notin \Pp(0,\infty)$.  However, if $\pp\in \Pp(0,\infty)$ with
$p^->1$, then $\cpp \in \Pp(0,\infty)$.  

The variable Lebesgue space $L^{p(\cdot)}(0,\infty)$ is the set of
measurable functions $f$ such that the modular
\[\varrho_{p(\cdot)}(f):=\int_0^\infty |f(x)|^{p(x)} dx<\infty.\]
This becomes a Banach function space when equipped with the Luxemburg
norm defined by
\[\|f\|_{p(\cdot)}:=\inf\{\mu>0: \varrho_{p(\cdot)}(f/\mu)\leq 1\}.\]
If $\pp=p$ is constant, then $L^\pp(0,\infty)=L^p(0,\infty)$ with
equality of norms.

For our results we need to impose a regularity condition on $\pp$
at $0$ and at infinity.  

\begin{defn} \label{defn:log-holder}
Given $\pp\in \Pp(0,\infty)$, 
we say that  $\pp$ is log-H\"older continuous at the origin, and
denote this by $\pp\in LH_0(0,\infty)$, if there exist constants
$C_0>0$ and $p_0\ge 1$ such that
\[|p(x)-p_0|\leq \frac{C_0}{-\log (x)},\quad \text{for all } 0<x< 1/2.\]
We say that $\pp$ is log-H\"older continuous at infinity, and denote
it by $\pp\in LH_\infty(0,\infty)$, if there exist constants
$C_\infty>0$ and $p_\infty\geq 1$ such that
\[|p(x)-p_\infty|\leq \frac{C_\infty}{\log(e+x)},\quad  \text{for all }x\in (0,\infty).\]
\end{defn}

Observe that if $\pp\in LH_0(0,\infty)$, then
$p_0=\displaystyle\lim_{x\to 0^+}p(x)$, which allows us to define
$p(0)=p_0$.  Similarly, if $\pp \in LH_\infty(0,\infty)$, then
$p_\infty=\displaystyle\lim_{x\to\infty}p(x)$.  Moreover, if $p_->1$
then it is easy to see that $\pp\in LH_0(0,\infty)$ and
$\pp\in LH_\infty(0,\infty)$ imply $\cpp \in LH_0(0,\infty)$ and
$\cpp\in LH_\infty(0,\infty)$ with $(p')_\infty=(p_\infty)'$,
respectively.

For many results in harmonic analysis to be true in the variable
Lebesgue spaces, it is necessary to assume a stronger condition than
$LH_0(0,\infty)$.  Instead, we assume that the exponent $\pp$ is log-H\"older continuous at
every point in $(0,\infty)$:
\begin{equation} \label{eqn:LH0}
  |p(x)-p(y)|\leq \frac{C_0}{-\log (|x-y|)},
\quad \text{for all } 0<|x-y|< 1/2.
\end{equation}
However, for the Hardy operator, it was shown that this condition is
not necessary, and the weaker condition $LH_0(0,\infty)$ is sufficient:
see~\cite{DS}.

Given a weight  $w$--i.e., a non-negative, locally integrable
function on $(0,\infty)$ such that $0<w(x)<\infty$ a.e.--we define the weighted variable Lebesgue
space $L^\pp_w(0,\infty)$  as follows:  
$f\in L^{p(\cdot)}_w(0,\infty)$ if $fw\in L^{p(\cdot)}(0,\infty)$.
When $\pp=p$ is constant, this becomes the weighted Lebesgue space
$L^p(w^p)$.  (In other words, in the variable Lebesgue spaces we
define weights as multipliers rather than as measures.)   

Given a weight $w$ and an 
operator $T$, we say that $T$ is strong-type $(\pp,\qq)$ with respect
to $w$ if 
\[\| (T f)w\|_{q(\cdot)}\le K\| fw\|_{p(\cdot)};\]
equivalently, $T : L^\pp_w(0,\infty) \rightarrow L^\qq_w(0,\infty)$.
We say that $T$ is weak-type $(\pp,\qq)$ with respect to $w$ if 
for all $\mu>0$, 
\[\mu \|w\chi_{\{x\in (0,\infty): Tf(x)>\mu\}}\|_{q(\cdot)}\leq
  K\|fw\|_{p(\cdot)}.\]
Note that if $T$ is strong-type $(\pp,\qq)$ with respect to $w$, then
it is automatically of weak-type as well.  

The weights we consider are a generalization of the $A_{p,q,0}$
weights defined above.

\begin{defn}\label{Apq0-var}
  Given $0<\lambda \leq 1$ and $\pp\in \Pp(0,\infty)$ such that $p_+<\frac{1}{1-\lambda}$, define $\qq$ by  $1/q(x)=1/p(x)-(1-\lambda)$. We say that a weight $w\in A_{p(\cdot), q(\cdot),0}$ if there
  exists a constant $C>0$ such that for every $b>0$,
\[\|w\chi_{(0,b)}\|_{q(\cdot)}\|w^{-1}\chi_{(0,b)}\|_{p'(\cdot)}\leq
  Cb^\lambda.\]
If $\lambda=1$, then $\pp=\qq$ and we  write $w\in A_{p(\cdot),0}$.
\end{defn}

The $A_{p(\cdot), q(\cdot),0}$ condition is a weaker version of the
class $A_{\pp,\qq}$
introduced in \cite{BDP} (see also \cite{CUW}) to control weighted
norm inequalities for the fractional integral operator.  Similarly,
the $A_{p(\cdot), 0}$ condition is a weaker version of the
$A_{p(\cdot)}$ condition \cite{CUDH,CUFN} which governs weighted norm
inequalities for the maximal operator on weighted Lebesgue spaces.
When $\pp$ and $\qq$ are constant, then the $A_{\pp,\qq,0}$ condition
becomes the $A_{p,q,0}$ condition defined above.


We can now state our main results.  The first is for the maximal
operator $N$.  

\begin{thm}\label{strongtypeN}
  Given $\pp \in \Pp(0,\infty)$, suppose $\pp\in LH_0(0,\infty)\cap
  LH_\infty(0,\infty)$ and
  $1<p^-\le p^+<\infty$.  If $w$ is a weight on $(0,\infty)$, then the
  following are equivalent:

\begin{enumerate}[label=(\roman*)]
\item The maximal operator $N$ is of strong-type $(p(\cdot),p(\cdot))$ with respect to $w$.
\item The maximal operator $N$ is of weak-type $(p(\cdot),p(\cdot))$ with respect to $w$.
\item $w\in A_{p(\cdot),0}$.
\end{enumerate}
\end{thm}

\begin{rem}
In the proof of Theorem~\ref{strongtypeN} we do not need to assume the
log-H\"older continuity conditions in order to prove the necessity of
the $A_{\pp,0}$ condition.  This raises the question of whether there
are weaker conditions on $\pp$ so that the $A_{\pp,0}$ condition is
also sufficient.  A similar question has been asked for the
Hardy-Littlewood maximal operator:  see~\cite{CUFN,L}.
\end{rem}

Theorem~\ref{strongtypeN} is the heart of our work.  Our
proof is adapted from the proof of the boundedness of the
Hardy-Littlewood maximal operator on weighted variable Lebesgue spaces
in~\cite{CUFN}. However, the fact that $N$ is an operator on the half-line
introduces a number of technical obstacles that were not present in
that proof.  

Given Theorem~\ref{strongtypeN} we can deduce the
following result that characterizes the weights controlling the
boundedness of the generalized Calder\'on operator $\mc C_\lambda$ and
the generalized Stieltjes transform $S_\lambda$ using a Hedberg type
inequality.

\begin{thm}\label{strongtypeS-C}
  Given  $0<\lambda \leq 1$ and $\pp \in \Pp(0,\infty)$, suppose 
  $\pp\in LH_0(0,\infty)\cap LH_\infty(0,\infty)$ and
  $1<p^-\le p^+<\frac{1}{1-\lambda}$.   Define $\qq\in \Pp(0,\infty)$  by
  $1/q(x)=1/p(x)-(1-\lambda)$. If $w$ is a weight on
  $(0,\infty)$, then  the following are equivalent:
\begin{enumerate}[label=(\roman*)]
\item The operator $\mc C_\lambda$ is of strong-type $(p(\cdot),q(\cdot))$ with respect to $w$.
\item The operator $S_\lambda$ is of strong-type $(p(\cdot),q(\cdot))$ with respect to $w$.
\item The operator $\mc C_\lambda$ is of weak-type $(p(\cdot),q(\cdot))$ with respect to $w$.
\item The operator $S_\lambda$ is of weak-type $(p(\cdot),q(\cdot))$ with respect to $w$.
\item $w\in A_{p(\cdot),q(\cdot),0}$.
\end{enumerate}
\end{thm}

As a consequence of Theorem~\ref{strongtypeS-C} we immediately get
weighted norm inequalities for $\mc C^\alpha$ and $S^\alpha$.  
Since $\mc C^\alpha\lesssim \mc C$, we  have that the
$A_{\pp,0}$ weights are sufficient for the 
boundedness of $\mc C^\alpha$ for
any $\alpha\geq 0$.  Surprisingly, $w\in A_{p(\cdot),0}$ is also
necessary, and it does not depend
on $\alpha$.
  
\begin{thm}\label{Calpha}
  Given $\pp \in \Pp(0,\infty)$, suppose  $\pp\in LH_0(0,\infty)\cap LH_\infty(0,\infty)$ and
  $1<p^-\leq p^+<\infty$.  If $w$ is a weight on $(0,\infty)$, then  the
  following statements are equivalent:
\begin{enumerate}[label=(\roman*)]
\item \label{wAp0}$w\in A_{p(\cdot),0}$.
\item There exists $\alpha\geq 0$ such that $\mc C^\alpha$ is of strong-type $(p(\cdot),p(\cdot))$ with respect to $w$.
\item For every $\alpha\geq 0$, $\mc C^\alpha$ is of strong-type $(p(\cdot),p(\cdot))$ with respect to $w$.
\item \label{weakCalpha}There exists $\alpha\geq 0$ such that $\mc C^\alpha$ is of weak-type $(p(\cdot),p(\cdot))$ with respect to $w$.
\item For every $\alpha\geq 0$, $\mc C^\alpha$ is of weak-type $(p(\cdot),p(\cdot))$ with respect to $w$.	
\end{enumerate}
Since $S^\alpha f\sim \mc C^\alpha f$, the same equivalence is true
with $\mc C^\alpha$ replaced by $S^\alpha$.
\end{thm}

Since we also have that $\mc C_\lambda = H_\lambda  + H_\lambda^*$, as
an immediate consequence of Theorem~\ref{strongtypeS-C} we get
weighted bounds for the Hardy operators.

\begin{thm}\label{thm:hardy-op}
  Given $0<\lambda \leq 1$ and $\pp \in \Pp(0,\infty)$, suppose 
  $\pp\in LH_0(0,\infty)\cap LH_\infty(0,\infty)$ and
  $1<p^-\le p^+<\frac{1}{1-\lambda}$.    Define $\qq\in \Pp(0,\infty)$  by
  $1/q(x)=1/p(x)-(1-\lambda)$. If $w$ is a weight on
  $(0,\infty)$, then  the following are equivalent:
\begin{enumerate}[label=(\roman*)]
\item The operators $H_\lambda$ and $H_\lambda^*$ are of strong-type $(p(\cdot),q(\cdot))$ with respect to $w$.
\item The operators $H_\lambda$ and $H_\lambda^*$ are of weak-type $(p(\cdot),q(\cdot))$ with respect to $w$.
\item $w\in A_{p(\cdot),q(\cdot),0}$.
\end{enumerate}
\end{thm}

One-weight norm inequalities for the Hardy operators in the variable Lebesgue spaces do not
appear to have been considered before now.  For two-weight
inequalities, see Mamedov, {\em et al.}~\cite{CUM,HM,
  MH09,MH10,MZ, MCMO}.  These results are not
immediately comparable to ours, even in the one-weight case, since
they assume log-H\"older continuity conditions that depend on the
weight.  See~\cite{CUM} for a discussion of cases where this
condition overlaps with our regularity assumptions.

\begin{rem} \label{rem:ctr-example}
It is  tempting to conjecture that either the strong or weak type
inequality for only one of the operators $H_\lambda$ or $H_\lambda^*$
implies the $A_{\pp,\qq,0}$ condition.  However, this is not true even
in the constant exponent case.  For simplicity we will show this when
$p=2$ and $\lambda=1$, but our example can easily be modified to work
for any $p$ and $\lambda$.   By~\cite{M1}, a necessary and
sufficient condition for $H_1$ to be bounded on $L^2(w)$ is that
\begin{equation} \label{eqn:Muck-suff}
\sup_{r>0} \int_r^\infty \frac{w(x)}{x^2}\,dx
\int_0^r w(x)^{-1}\,dx < \infty.  
\end{equation}
Let 
\[ w(x) =
\begin{cases}
1 & 0 < x \leq 1 \\
e^{-x} & x>1.
\end{cases}
\]
This weight satisfies \eqref{eqn:Muck-suff}.  Indeed, if $r\leq 1$, then 
\[ \int_r^\infty \frac{w(x)}{x^2}\,dx
\int_0^r w(x)^{-1}\,dx
\leq
\int_r^\infty \frac{dx}{x^2} \int_0^r dx = \frac{1}{r}\cdot r = 1. \]
And if $r>1$, the left-hand side is dominated by
\[ \int_r^\infty \frac{e^{-x}}{x^2}\,dx 
\int_0^r e^x\,dx 
\leq 
\frac{e^{-r}}{r} e^r \leq 1. \]

On the other hand, $w\not\in A_{2,0}$, since for every $r>1$,
\[ \frac{1}{r}\int_0^r w(x)\,dx \frac{1}{r}\int_0^r w(x)^{-1}\,dx
\geq
\frac{1}{r}\int_1^r e^{-x}\,dx \frac{1}{r}\int_1^r e^x\,dx 
=
\frac{e^{-1}-e^{-r}}{r^2} (e^r-1), \]
and the right-hand side is unbounded as $r\rightarrow \infty$. 

For
a related instance in which the $A_{p,0}$ condition is sufficient but
not necessary, see~\cite{AM}.
\end{rem}

\medskip

We now give two applications of Theorem~\ref{strongtypeS-C}.  More
precisely, we will give an application of a generalization of this
theorem to higher dimensions.    If we replace $(0,\infty)$ by $\R^n$,
then we may define the variable Lebesgue space $\Lp(\R^n)$ exactly as
above.  We define log-H\"older continuity as in
Definition~\ref{defn:log-holder}, replacing $x$ by $|x|$ on the
right-hand side of each inequality.  Finally, we say that a weight $w\in A_{\pp,0}$ if for all $b>0$,
\[\|w\chi_{B(0,b)}\|_{p(\cdot)}\|w^{-1}\chi_{B(0,b)}\|_{p'(\cdot)}\leq Cb^n.\]

For a measurable function $f$ on
$\R^n$, define the radial operators
\begin{equation}\label{NRn} Nf(x) = \sup_{b>|x|} \frac{1}{b^n}\int_{B(0,b)}|f(y)|\,dy, 
\end{equation}
and
\[Sf(x)=\int_{\mbb R^n} \frac{f(y)}{|x|^n+|y|^n}\,dy, \quad x\in \mbb{R}^n.\]
Then we can modify the proofs of Theorems~\ref{strongtypeN}
and~\ref{strongtypeS-C} to get the following result.

\begin{thm} \label{thm:n-dim}
Given $\pp \in \Pp(\R^n)$, suppose $\pp\in LH_0(\R^n)\cap
LH_\infty(\R^n)$ and $1<p_-\leq p_+<\infty$.  If  $w$ is a weight on
$\R^n$, then the following are equivalent:
\begin{enumerate}[label=(\roman*)]
\item $N$ is strong $(\pp,\pp)$ with respect to $w$;
\item $S$ is strong $(\pp,\pp)$ with respect to $w$;
\item $w\in A_{\pp,0}$.
\end{enumerate}
\end{thm}

The first application of Theorem~\ref{thm:n-dim} is a weighted version
of Hilbert's inequality:  
for $p>1$ and non-negative functions $f,g$, 
\[\int_0^\infty \int_0^\infty \frac{f(x)g(y)}{x+y}\,dxdy
\leq 
C_p\|f\|_{L^p(0,\infty)}\|g\|_{L^{p'}(0,\infty)},\]
which was first proved by G. Hardy and M. Riesz~\cite{H} (also
see~\cite[Chapter IX]{HLP}).

\begin{thm}\label{Hilbertineq}
  Given $\pp\in \Pp(\R^n)$, suppose
  $\pp\in LH_0(\mbb{R}^n)\cap LH_\infty (\mbb{R}^n)$ and
  $1<p^-\leq p^+<\infty$. Then there exists $C>0$ such that for any non-negative functions
  $f,\,g$, 
  $f\in L^{p(\cdot)}_w(\mbb{R}^n)$ and   $g\in L^{p'(\cdot)}_{w^{-1}}(\mbb{R}^n)$, 
  independent of $f$ and $g$,
	\begin{equation}\label{Hilbert}
	\int_{\mbb R^n} \int_{\mbb R^n}
        \frac{f(x)g(y)}{|x|^n+|y|^n}\,dxdy
\leq 
C \|fw\|_{p(\cdot)}\|gw^{-1}\|_{p'(\cdot)}
	\end{equation} 
 if and only if $w\in A_{p(\cdot),0}$.
\end{thm}

Theorem~\ref{Hilbertineq} appears to be new, even in the constant
exponent case.  When $n=1$ it is implicit in~\cite{DMRO1}.

\begin{rem}
The sharp constant in Hilbert's inequality is
$\frac{\pi}{\sin(\pi/p)}$; this is due to J. Schur~\cite{Sch}. Here,
we are  not concerned with finding the best
constant.  However, this is an interesting problem, especially in the
constant exponent case where there has been a great deal of work on
sharp constants
related to the so-called $A_2$ conjecture.  See, for
instance,~\cite{Hy}. 
\end{rem}

The second application of Theorem~\ref{thm:n-dim} is to the
continuity of certain integral operators on variable Lebesgue spaces.
Given an index set $J$, let $\{T_j\}_{j\in J}$ be a family of
(singular) integral
operators defined by
\[T_jf(x)=pv \int_{\mbb{R}^n}K_j(x,y)f(y)\, dy\]
where each $K_j$ satisfies a decay estimate,
\begin{equation}\label{boundKj}
|K_j(x,y)|\leq \frac{C_0}{|x-y|^n}, \quad x\neq y,
\end{equation}
with $C_0$ independent of $j\in J$.  We are interested in the
boundedness of the associated maximal operator 
\[T^*f(x)=\sup\limits_{j\in J} |T_jf(x)|.\]

These operators were first considered by Soria and Weiss
in \cite{SW}.  They prove that $T^* : L^p(w) \rightarrow L^p(w)$
provided that $w$ is an $A_p$ weight that is 
essentially constant over dyadic
annuli.  More precisely, they assume that
there exists a constant $C_1>0$ such that
\begin{equation}\label{constannuli}
\sup\limits_{2^{k-2}\leq |x|\leq 2^{k+1}} w(x)
\leq C_1\inf\limits_{2^{k-2}\leq |x|\leq 2^{k+1}} w(x),
\quad k\in \mbb{Z}.
\end{equation}

We can extend their result to the variable Lebesgue spaces.

\begin{thm}\label{SoriaWeiss}
Let $\{T_j\}_{j\in J}$, $T^*$ be defined as above.  Given $\pp \in
\Pp(\R^n)$ suppose  $\pp\in LH_0(\mbb{R}^n)\cap LH_\infty (\mbb{R}^n)$,
$1<p^-\leq p^+<\infty$, and for every family of balls $\mc{B}$ with
bounded overlap,
\begin{equation}\label{classG}
	\sum\limits_{B\in \mc B}
        \|f\chi_B\|_{p(\cdot)}\|g\chi_B\|_{p'(\cdot)}
\leq 
C \|f\|_{p(\cdot)}\|g\|_{p'(\cdot)},
\end{equation}
where the constant $C$ is independent of $\mc{B}$ and only depends on
$\pp$ and the bound on the overlap. If $T^*$ is of strong type
$(p(\cdot),p(\cdot))$, and if $w\in A_{p(\cdot),0}$ and
satisfies~\eqref{constannuli}, then $T^*$ is of strong type
$(p(\cdot),p(\cdot))$ with respect to $w$.
\end{thm}

Theorem~\ref{SoriaWeiss} is new, but this question has also been
considered by Bandaliev \cite{B,B2}.  However, his results
have  different hypotheses
on $\pp$ and the weights, and his proofs rely on  other techniques.  

\begin{rem}
  The summation condition \eqref{classG} was introduced by
  Berezhnoi~\cite{Ber} in the study of Banach function spaces.
  In~\cite{DHHR} this condition was shown to be very closely related
  to the boundedness of the Hardy-Littlewood maximal operators and
  singular integrals on the variable Lebesgue spaces.  Thus it is a
  very reasonable assumption in the context of
  Theorem~\ref{SoriaWeiss}.  As shown in~\cite[Theorem~7.3.22]{DHHR}, this condition holds if $\pp\in LH(\R^n)\cap LH_\infty(\R^n)$, where
  $LH(\R^n)$ is the local log-H\"older condition defined
  by~\eqref{eqn:LH0}.
\end{rem}

\begin{rem}
In the constant exponent case, Theorem~\ref{SoriaWeiss} appears to be
a generalization of the original result of Soria and Weiss, since we
only assume $w\in A_{p,0}$ whereas they assume the stronger condition
$w\in A_p$.  However, given the additional
assumption~\eqref{constannuli},
these two conditions are the same: Clearly, we always have
$A_p\subset A_{p,0}$.  Conversely, given $w\in A_{p,0}$ that
satisfies~\eqref{constannuli}, fix any ball $B=B(x,r)$.  If
$r>|x|/2$, then $B\subset \bar{B}= B(0,s)$, $s=|x|+r$, and 
$|B|\sim |\bar{B}|$.  Hence,
\[ \frac{1}{|B|}\int_B w\,dx 
\left(\frac{1}{|B|}\int_B w^{1-p'}\,dx\right)^{p-1}
\lesssim 
\frac{1}{|\bar{B}|} \int_{\bar{B}} w\,dx 
\left(\frac{1}{|\bar{B}|} \int_{\bar{B}}
    w^{1-p'}\,dx\right)^{p-1} \leq C. \]
On the other hand, if $r\leq |x|<2$, and $k\in \mathbb Z$ is such that $2^{k-1}\leq |x| <2^{k}$, then
for any $y \in B$, $2^{k-2}\leq |y| <2^{k+1}$, and so $w$ is
essentially constant on $B$, so the $A_p$ condition holds on $B$.
\end{rem}

The remainder of this paper is organized as follows.  In Section
\ref{lemmas} we state and prove a number of technical lemmas on the
exponents $p(\cdot)$ and the weights $A_{p(\cdot), q(\cdot),0}$ that
we will use in the proofs of our main results.  The proof of Theorem
\ref{strongtypeN} is in Section \ref{secN}, and the proofs of
Theorems~\ref{strongtypeS-C} and~\ref{Calpha} are in
Section~\ref{secC-Calpha}.  Finally, Section \ref{app} contains the
proof of Theorems~\ref{thm:n-dim},~\ref{Hilbertineq}
and~\ref{SoriaWeiss}.

\section{Technical results}\label{lemmas}

In this section we establish some properties of log-H\"older
continuous exponents and $A_{\pp,0}$ weights that we will use in our
main proofs.  We begin with two lemmas that allow us to apply the
$LH_0$ and $LH_\infty$ conditions.  The first is a version of
\cite[Lemma~3.2]{Diening} (see also \cite{CUFbook}) to the basis of
intervals $\{(0,b)\}_{b>0}$.

\begin{lem}\label{equivPlog0}
  Given $\pp \in \Pp(0,\infty)$, suppose $\pp\in LH_0(0,\infty)$.
  Then there exists $C>0$ such that for every $b>0$,
\[b^{p^-_{(0,b)}-p^+_{(0,b)}}\leq C.\]
\end{lem}

\begin{proof}
Fix  $\pp\in LH_0(0,\infty)$. Since $p^-_{(0,b)}-p^+_{(0,b)}\leq 0$, we can
  assume that   $0<b<1/2$. For if $b \geq 1/2$, then
\[b^{p^-_{(0,b)}-p^+_{(0,b)}}\leq (1/2)^{p^-_{(0,b)}-p^+_{(0,b)}}\leq (1/2)^{p^--p^+}=2^{p^+-p^-}.\]

Fix $0<b<1/2$. We will bound the difference
$p^-_{(0,b)}-p^+_{(0,b)}$. From the definition of $p^-_{(0,b)}$, given
any $\epsilon>0$, there exists $0<x_\epsilon<b<1/2$ such that
$0\leq p(x_\epsilon)-p^-_{(0,b)}<\epsilon$. Consequently,
\[0\leq p(0)-p^-_{(0,b)}\leq
  |p(0)-p(x_\epsilon)|+p(x_\epsilon)-p^-_{(0,b)}<\frac{C_0}{-\log(x_\epsilon)}+\epsilon\leq
  \frac{C_0}{-\log(b)}+\epsilon, \] 
and if we let $\epsilon\rightarrow 0$, we get
\[0\leq p(0)-p^-_{(0,b)}\leq \frac{C_0}{-\log(b)}.\]
Similarly, we have that 
\[0\leq p^+_{(0,b)}-p(0)\leq \frac{C_0}{-\log(b)}.\] 
Therefore,
\[0\leq p^+_{(0,b)}-p^-_{(0,b)}\leq \frac{2C_0}{-\log(b)}.\]
Now, since $1/b>2$,
\[b^{p^-_{(0,b)}-p^+_{(0,b)}}
=
(1/b)^{p^+_{(0,b)}-p^-_{(0,b)}}
\leq (1/b)^{\frac{2C_0}{-\log(b)}}
=b^{\frac{2C_0}{\log(b)}}=e^{2C_0}.\]
If we take $C=\max\{2^{p^+-p^-}, e^{2C_0}\}$, we get the desired inequality.
\end{proof}

The next result allows us to estimate the modular
$\varrho_{p(\cdot)}(f)$ by means of the modular
$\varrho_{p_\infty}(f)$ whenever $\pp\in LH_\infty(0,\infty)$. This
result is from~\cite[Lemma~2.7]{CUFN}, but as they noted there, the
proof is identical to the case with Lebesgue
measure~\cite[Lemma~3.26]{CUFbook}.

\begin{lem}\label{changeexp}
Given $\pp \in \Pp(0,\infty)$, suppose  $\pp\in LH_\infty(0,\infty)$.
Fix a set $G\subset (0,\infty)$ and  a non-negative measure $\mu$. Then, for every
  $t>1/p^-$, there exists a positive constant $C_t=C(t,C_\infty)$ such
  that for all functions $g$ with $0\leq g(x)\leq 1$,
 \[\int_G g(x)^{p(x)}d\mu(x)
\leq C_t\int_G g(x)^{p_\infty} d\mu(x) 
+\int_G \frac{1}{(e+x)^{tp^-_G}} d\mu(x),\]
 and
 \[\int_G g(x)^{p_\infty}d\mu(x)
\leq C_t\int_G g(x)^{p(x)} d\mu(x) 
+\int_G \frac{1}{(e+x)^{tp^-_G}} d\mu(x).\]
\end{lem}

\medskip

In the next series of results, we establish the properties of
$A_{p(\cdot),0}$ weights.  These are similar to the properties of the
$A_\pp$ weights established in~\cite[Section~3]{CUFN}, which in turn
are related to the properties of the Muckenhoupt $A_p$ weights.

\begin{lem}\label{Ainftynorm}
  Given $\pp\in\Pp(0,\infty)$, if $w\in A_{p(\cdot),0}$, then there exists $C>0$
  such that for any $b>0$ and any measurable set $E\subset (0,b)$,
\[\frac{|E|}{b}=\frac{|E|}{|(0,b)|}
\leq C \frac{\|w\chi_E\|_{p(\cdot)}}{\|w\chi_{(0,b)}\|_{p(\cdot)}}.\]
\end{lem}

\begin{proof}
  Fix $b>0$ and $E\subset (0,b)$. Then by H\"older's inequality and
  the $A_{p(\cdot),0}$ condition we have
\[|E|=\int_0^b w(x)\chi_E(x) w^{-1}(x)\,dx
\leq C\|w\chi_E\|_{p(\cdot)}\|w^{-1}\chi_{(0,b)}\|_{p'(\cdot)}
\leq
\frac{Cb\|w\chi_E\|_{p(\cdot)}}{\|w\chi_{(0,b)}\|_{p(\cdot)}}.\qedhere\]
\end{proof}

\begin{lem}\label{Plog0norm}
  Given $\pp \in \Pp(0,\infty)$ suppose $\pp\in LH_0(0,\infty)$.  If
  $w\in A_{p(\cdot),0}$,  then there exists $C_0>0$, depending on $\pp$
  and $w$, such that for every $b>0$,
\[\|w\chi_{(0,b)}\|_{p(\cdot)}^{p^-_{(0,b)}-p^+_{(0,b)}}\leq C_0.\]
\end{lem}

\begin{proof}Fix $b>0$. We will consider two cases: $b<1$ and $b\geq 1$.

If $b<1$, then we apply the previous lemma with $E=(0,b)\subset (0,1)$
to get
\[\|w\chi_{(0,b)}\|_{p(\cdot)}\geq Cb\|w\chi_{(0,1)}\|_{p(\cdot)}.\]
Then by Lemma \ref{equivPlog0},
\begin{align*}
\|w\chi_{(0,b)}\|_{p(\cdot)}^{p^-_{(0,b)}-p^+_{(0,b)}}
&\leq (Cb\|w\chi_{(0,1)}\|_{p(\cdot)})^{p^-_{(0,b)}-p^+_{(0,b)}}\\
&\leq Cb^{p^-_{(0,b)}-p^+_{(0,b)}}(1+\|w\chi_{(0,1)}\|_{p(\cdot)}^{-1})^{p^+_{(0,b)}-p^-_{(0,b)}}\\
&\leq C(1+\|w\chi_{(0,1)}\|_{p(\cdot)}^{-1})^{p^+-p^-}=C_1.
\end{align*}

If $b\geq 1$, then we repeat the argument but now take
$E=(0,1)\subset (0,b)$ and use Lemma \ref{Ainftynorm} with
$w^{-1}\in A_{p'(\cdot),0}$. By H\"older's inequality,
\[\|w^{-1}\chi_{(0,1)}\|_{p'(\cdot)}
\geq \frac{C\|w^{-1}\chi_{(0,b)}\|_{p'(\cdot)}}{b}
\geq \frac{C}{\|w\chi_{(0,b)}\|_{p(\cdot)}}.\]
Thus,
\begin{multline*}
\|w\chi_{(0,b)}\|_{p(\cdot)}^{p^-_{(0,b)}-p^+_{(0,b)}}
\leq C\|w^{-1}\chi_{(0,1)}\|_{p'(\cdot)}^{p^+_{(0,b)}-p^-_{(0,b)}}\\
\leq C(1+\|w^{-1}\chi_{(0,1)}\|_{p'(\cdot)})^{p^+-p^-}=C_2.
\end{multline*}
If we let $C=\max\{C_1,C_2\}$ we get the desired inequality.
\end{proof}

\begin{rem}
  In Lemma~\ref{Plog0norm} the $LH_\infty(0,\infty)$ condition on
  $\pp$ is not required (as in \cite[Lemma~3.3]{CUFN}) since the
  intervals involved in the $A_{p(\cdot),0}$ condition are nested.
\end{rem}

We now want to define a condition analogous to the $A_\infty$
condition but associated with the basis of intervals $\{(0,b)\}_{b>0}$
(as considered in \cite{DMRO2}).  Hereafter, given an
exponent $\pp$ and a weight $w$, we define the weight $W(x)=w(x)^{p(x)}$
and denote $W(E)=\int_E W(x) \,dx$ for any measurable set
$E\subset(0,\infty)$. Similarly, for the dual weight $w^{-1}$ we write
$\sigma(x)=w(x)^{-p'(x)}$ and
$\sigma(E)=\int_E \sigma(x)\,dx$.

\begin{defn}
  Given a weight $w$ such that $0<w(0,b)<\infty$ for every $b>0$, we say that $w\in A_{\infty,0}$ if there exist
  constants $C,\delta>0$ such that for every $b>0$ and each measurable
  set $E\subset (0,b)$,
\[\frac{|E|}{b}\leq C\left(\frac{w(E)}{w(0,b)}\right)^{\delta}.\]
\end{defn}

As an immediate consequence of this definition, we have the following
lemma.

\begin{lem}\label{Ainftymedidas}
  If $w\in A_{\infty,0}$, for every $0<\alpha<1$, there exists
  $0<\beta<1$ (depending on $\alpha$) such that, given $b>0$ and a
  measurable set $E\subset (0,b)$, if $|E| \geq\alpha b$, then
  $w(E) \geq\beta w(0,b)$.
\end{lem}

The next lemma requires the deeper properties of the $A_\infty$
condition defined with respect to a basis.  

\begin{lem} \label{lemma:non-integrable}
If $w\in A_{\infty,0}$, then $w\not\in L^1(0,\infty)$.  
\end{lem}

\begin{proof}
It follows from~\cite[Theorems~3.1,~4.1]{DMRO2} that if $w\in
A_{\infty,0}$, then there exist constants $\gamma,\,\delta>1$, such
that for any $b>0$, if $E\subset (0,b)$ and $\gamma|E| \leq b$, then $\delta w(E)\leq w(0,b)$.  
In particular, if we let $b_k = \gamma^k$ for $k\in \mathbb N$, and let $E=(0,1)$, then
\[ w(0,b_k) \geq \delta^k w(0,1).  \]
Since the right-hand side tends to infinity as $k\rightarrow \infty$ (recall that $0<w(0,1)<\infty$),
we get the desired conclusion.
\end{proof}

We will apply these lemmas to the weights $W$ and
$\sigma$ using the following result.

\begin{lem}\label{Ainfty} 
Given $\pp\in \Pp(0,\infty)$, suppose  $\pp\in LH_0(0,\infty)\cap
LH_\infty(0,\infty)$.
 If $w\in A_{p(\cdot),0}$, then $W\in A_{\infty,0}$.
\end{lem}

\begin{proof}Notice first that from the fact that $0<w(x)<\infty$ a.e. and the $A_{p(\cdot),0}$ condition, $0<\|w\chi_{(0,b)}\|_{p(\cdot)}<\infty$ for every $b>0$. Hence, $0<W(0,b)<\infty$ for every $b>0$.
	
  Fix $b>0$ and a measurable set $E\subset (0,b)$. We consider three cases:
  $\|w\chi_E\|_{p(\cdot)}\leq \|w\chi_{(0,b)}\|_{p(\cdot)}\leq 1$,
  $\|w\chi_E\|_{p(\cdot)}\leq 1 \leq \|w\chi_{(0,b)}\|_{p(\cdot)}$ and
  $1<\|w\chi_E\|_{p(\cdot)}\leq \|w\chi_{(0,b)}\|_{p(\cdot)}$.

  In the first case, by \cite[Corollary 2.23]{CUFbook}, we have that
  $W(E)\leq W(0,b)\leq 1$,
  $\|w\chi_E\|_{p(\cdot)}\leq W(E)^{1/p^+_E}\leq W(E)^{1/p^+_{(0,b)}}$
  and $\|w\chi_{(0,b)}\|_{p(\cdot)}\geq W(0,b)^{1/p^-_{(0,b)}}$. Thus,
  by  Lemmas \ref{Ainftynorm} and \ref{Plog0norm} we get
\begin{align*}
\frac{|E|}{b}
&\leq C\frac{\|w\chi_E\|_{p(\cdot)}}{\|w\chi_{(0,b)}\|_{p(\cdot)}} \\
& =C\frac{\|w\chi_E\|_{p(\cdot)}}
{\|w\chi_{(0,b)}\|_{p(\cdot)}^{p^-_{(0,b)}/p^+_{(0,b)}}\|w\chi_{(0,b)}\|_{p(\cdot)}^{1-p^-_{(0,b)}/p^+_{(0,b)}}}\\
&\leq C\frac{W(E)^{1/p^+_{(0,b)}}
\big(\|w\chi_{(0,b)}\|_{p(\cdot)}^{p^-_{(0,b)}-p^+_{(0,b)}}\big)^{1/p^-}}{W(0,b)^{1/p^+_{(0,b)}}}\\
&\leq C\left(\frac{W(E)}{W(0,b)}\right)^{1/p^+_{(0,b)}} \\
&\leq C\left(\frac{W(E)}{W(0,b)}\right)^{1/p^+}.
\end{align*}

In the second case, if
$\|w\chi_E\|_{p(\cdot)}\leq 1 \leq \|w\chi_{(0,b)}\|_{p(\cdot)}$, then we have
$\|w\chi_E\|_{p(\cdot)}\leq W(E)^{1/p^+_{(0,b)}}$ and
$\|w\chi_{(0,b)}\|_{p(\cdot)}\geq W(0,b)^{1/p^+_{(0,b)}}$, which
yields
\[\frac{|E|}{b}
\leq C\frac{\|w\chi_E\|_{p(\cdot)}}{\|w\chi_{(0,b)}\|_{p(\cdot)}}
\leq C\left(\frac{W(E)}{W(0,b)}\right)^{1/p^+_{(0,b)}}
\leq C\left(\frac{W(E)}{W(0,b)}\right)^{1/p^+},\]
where we have used again Lemma \ref{Ainftynorm}.

Finally, in  the third case, if $1<\|w\chi_E\|_{p(\cdot)}\leq
\|w\chi_{(0,b)}\|_{p(\cdot)}$, then we will show that
\begin{equation}\label{case3}
\frac{|E|}{b}\leq C\left(\frac{W(E)}{W(0,b)}\right)^{1/p_\infty}
\leq C\left(\frac{W(E)}{W(0,b)}\right)^{1/p^+}.
\end{equation}
Since $\pp\in LH_\infty(0,\infty)$ and
$\|w\chi_{(0,b)}\|_{p(\cdot)}^{-1}\leq 1$, 
we can apply Lemma \ref{changeexp} with measure $d\mu(x)=w(x)^{p(x)}
\,dx$, $G=(0,b)$ and $g(x)\equiv
\|w\chi_{(0,b)}\|_{p(\cdot)}^{-1}$. Hence, for every $t>1/p^-$,
\[\int_0^b \frac{1}{\|w\chi_{(0,b)}\|_{p(\cdot)}^{p_\infty}}
  w(x)^{p(x)}\,dx
\leq C_t\int_0^b
\left(\frac{w(x)}{\|w\chi_{(0,b)}\|_{p(\cdot)}}\right)^{p(x)}\,dx
+\int_0^b \frac{w(x)^{p(x)}}{(e+x)^{tp^-}} \,dx.\]
By the definition of the norm the first term is equal to $C_t$. We
will now show that we can choose $t>1/p^-$, depending only on $\pp$
and $w$, such that the second term is smaller than 1. In fact,
\begin{align}\label{errorestimate}
\nonumber\int_0^\infty \frac{w(x)^{p(x)}}{(e+x)^{tp^-}} \,dx 
&=\int_0^1 \frac{w(x)^{p(x)}}{(e+x)^{tp^-}} \,dx
+\sum\limits_{k=0}^{\infty} \int_{2^k}^{2^{k+1}} \frac{w(x)^{p(x)}}{(e+x)^{tp^-}} \,dx\\
\nonumber&\leq \frac{W(0,1)}{e^{tp^-}} 
+\sum\limits_{k=0}^{\infty} \frac{1}{2^{ktp^-}}W(2^k,2^{k+1})\\
\nonumber&\leq \frac{W(0,1)}{e^{tp^-}}
+\sum\limits_{k=0}^{\infty} \frac{1}{2^{ktp^-}}W(0,2^{k+1})\\
&\leq \frac{W(0,1)}{e^{tp^-}}
+\sum\limits_{k=0}^{\infty} \frac{1}{2^{ktp^-}}\max\left\{\|w\chi_{(0,2^{k+1})}\|_{p(\cdot)}^{p^-},\|w\chi_{(0,2^{k+1})}\|_{p(\cdot)}^{p^+}\right\},
\end{align}
where in the last inequality we used~\cite[Corollary 2.23]{CUFbook}.

To estimate the norm $\|w\chi_{(0,2^{k+1})}\|_{p(\cdot)}$ we use Lemma
\ref{Ainftynorm} with $E=(0,1)\subset (0,2^{k+1})$:
\[\|w\chi_{(0,2^{k+1})}\|_{p(\cdot)}
\leq C\frac{|(0,2^{k+1})|}{|(0,1)|}\|w\chi_{(0,1)}\|_{p(\cdot)}\leq
C2^{k+1}.\]
Thus,
$\max\left\{\|w\chi_{(0,2^{k+1})}\|_{p(\cdot)}^{p^-},\|w\chi_{(0,2^{k+1})}\|_{p(\cdot)}^{p^+}\right\}
\leq
C2^{(k+1)p^+}\leq C2^{kp^+}$;  consequently,
\[\int_0^\infty \frac{w(x)^{p(x)}}{(e+x)^{tp^-}} \,dx\leq \frac{W(0,1)}{e^{tp^-}}+C\sum\limits_{k=0}^{\infty} \frac{2^{kp^+}}{2^{ktp^-}}.\]
If we take $t>p^+/p_-$, the last sum converges; hence, by the
dominated convergence theorem,
\[\lim_{t\rightarrow\infty} \int_0^\infty
  \frac{w(x)^{p(x)}}{(e+x)^{tp^-}} \,dx=0. \]
Furthermore,
$\displaystyle\lim_{t\rightarrow\infty} W(0,1)e^{-tp^- }=0$.
Therefore, we can take $t$ sufficiently large  that
\eqref{errorestimate} is less than 1. Therefore,
\[W(0,b)\leq (C_t+1)\|w\chi_{(0,b)}\|_{p(\cdot)}^{p_\infty}\]
or, equivalently,
\begin{equation}\label{eq1}
W(0,b)^{1/p_\infty}\leq (C_t+1)^{1/p_\infty}\|w\chi_{(0,b)}\|_{p(\cdot)}.
\end{equation}

We now estimate the term $W(E)$.  We  again apply
Lemma~\ref{changeexp}, exchanging the roles of $\pp$ and
$p_\infty$. Thus,
\[1=\int_E \left(\frac{w(x)}{\|w\chi_E\|_{p(\cdot)}}\right)^{p(x)}\,dx
\leq C_t \int_E \|w\chi_E\|_{p(\cdot)}^{-p_\infty}w(x)^{p(x)}
+\int_E \frac{w(x)^{p(x)}}{(e+x)^{tp^-}} \,dx.\]
If we repeat the above argument, we can make the last integral smaller
than $1/2$, which gives us,
\begin{equation}\label{eq2}
\|w\chi_E\|_{p(\cdot)}^{p_\infty}\leq 2C_t W(E).
\end{equation}
If we combine \eqref{eq1} and \eqref{eq2}, we get \eqref{case3}. This
completes the proof.
\end{proof}

From inequalities \eqref{eq1} and \eqref{eq2} with $E=(0,b)$, we get the following corollary.

\begin{cor}\label{normmodularpinfty}
Given $\pp \in \Pp(0,\infty)$ suppose $\pp\in LH_0(0,\infty)\cap
LH_\infty(0,\infty)$.
 If $w\in A_{p(\cdot),0}$ and $b>0$ such that
 $\|w\chi_{(0,b)}\|_{p(\cdot)}\geq 1$, then
\[\|w\chi_{(0,b)}\|_{p(\cdot)}\sim W(0,b)^{1/p_\infty}.\]
\end{cor}

\section{Proof of Theorem \ref{strongtypeN}}\label{secN}

\begin{proof}

  The implication \textit{(i)}$\Rightarrow$\textit{(ii)} is
  straightforward. We will next prove 
  \textit{(ii)}$\Rightarrow$\textit{(iii)}. Suppose that for every
  $\mu>0$ and every $f\in L^{p(\cdot)}_w(0,\infty)$,
\[\mu\|w\chi_{\{x\in (0,\infty): Nf(x)>\mu\}}\|_{p(\cdot)}
\leq K\|fw\|_{p(\cdot)}.\]
Fix $b>0$; then by duality there exists  a non-negative function $g\in
L^{p(\cdot)}(0,\infty)$ such that $\|g\|_{p(\cdot)}\leq 1$ and
\[\|w^{-1}\chi_{(0,b)}\|_{p'(\cdot)}\sim \int_0^b w^{-1}(y)g(y)\,dy.\]
Without loss of generality, we may suppose that
$\|w^{-1}\chi_{(0,b)}\|_{p'(\cdot)}>0$. If we let
$f=\chi_{(0,b)}w^{-1}g$ and
$\mu=\frac{1}{b}\int_0^b w^{-1}(y)g(y)\,dy>0$, then for every
$x\in (0,b)$, $Nf(x)\geq \mu$. Thus, for every $\nu>1$,
$(0,b)\subset \{x\in (0,\infty): Nf(x)>\mu/\nu\}$.  From the weak-type
inequality, if we let $\nu\rightarrow 1$,
\[\left(\frac{1}{b}\int_0^b w^{-1}(y)g(y)\,dy\right)
\|w\chi_{(0,b)}\|_{p(\cdot)}
\leq C K\|fw\|_{p(\cdot)}
=C\|g\|_{p(\cdot)}\leq C.\]
Therefore, we have that 
\begin{multline*}
\|w\chi_{(0,b)}\|_{p(\cdot)}\|w^{-1}\chi_{(0,b)}\|_{p'(\cdot)}
\sim\|w\chi_{(0,b)}\|_{p(\cdot)} \int_0^b w^{-1}(y)g(y)\,dy\\
=C b\|w\chi_{(0,b)}\|_{p(\cdot)} \left(\frac{1}{b}\int_0^b w^{-1}(y)g(y)\,dy\right)
\leq Cb.
\end{multline*}
Since this is true for all $b>0$, $w\in A_{\pp,0}$.

\medskip

We now come to the proof of \textit{(iii)}$\Rightarrow$\textit{(i)},
which is the most difficult part.  Fix $w\in A_{p(\cdot),0}$; without
loss of generality we may assume that $f\geq 0$ and
$\|fw\|_{p(\cdot)}\leq 1$. We begin by arguing as in the proof of
\cite[Lemma~2.2]{DMRO1}. From the definition we have that $Nf$ is
decreasing and continuous. Thus, given $\mu>0$, if the level set
$\{x\in (0,\infty): Nf(x)>\mu\}\neq \emptyset$, it either equals $(0,b)$ for
some $b>0$ or it equals $(0,+\infty)$. In the first case, we have that
\[\lambda b=\int_0^b f(x) \,dx\]
while in the second case,
\[\int_0^\infty f(x) \,dx=+\infty.\]
To avoid the latter case, we shall further assume that $f$ is bounded
and has compact support.  The full result then follows by a standard
density argument (cf.~\cite[Section~3.4]{CUFbook}).

We now split
$f=f_1+f_2$, where $f_1=f\chi_{\{f\sigma^{-1}>1\}}$ and
$f_2=f\chi_{\{f\sigma^{-1}\leq 1\}}$. Then, $Nf\leq Nf_1+Nf_2$ and
\begin{equation} \label{eqn:modular}
\int_0^\infty f_i(x)^{p(x)}w(x)^{p(x)} \,dx
\leq \int_0^\infty f(x)^{p(x)}w(x)^{p(x)} \,dx\leq 1, 
\quad i=1,2.
\end{equation}
Hence, it will suffice to show that
\begin{equation}\label{Nfi}
I_i:=\int_0^\infty Nf_i(x)^{p(x)} w(x)^{p(x)}\,dx\leq C,\quad i=1,2.
\end{equation}

\textit{Estimate for $I_1$:} By our choice of $f$, we can find a
non-increasing sequence of positive real numbers
$\{b_k\}_{k\in \mbb{Z}}$ such that
$\{x\in (0,\infty): Nf_1(x)>2^k\}=(0,b_k)$,
$\{x\in (0,\infty):2^k<Nf_1(x)\leq 2^{k+1}\}=[b_{k+1},b_k)$ and
\[2^k b_k=\int_0^{b_k} f_1(x) \,dx.\]
Consequently, we have that $b_{k+1}\leq b_k/2$, and so
$|[b_{k+1},b_k)|\geq |(0,b_k)|/2$. For simplicity, from now on we will
write $p^-_{(0,b_k)}=p^-_k$ and $p^+_{(0,b_k)}=p^+_k$.

Given this decomposition, we estimate \eqref{Nfi} by adapting the
approach in~\cite{CUFN}:
\begin{align}\label{decomp}
I_1&=\sum\limits_{k\in \mbb{Z}} \int_{b_{k+1}}^{b_k} Nf_1(x)^{p(x)} w(x)^{p(x)}\,dx\\
\nonumber&\leq \sum\limits_{k\in \mbb{Z}} \int_{b_{k+1}}^{b_k} (2^{k+1})^{p(x)} w(x)^{p(x)}\,dx\\
\nonumber&\leq 2^{p^+}\sum\limits_{k\in \mbb{Z}} \int_{b_{k+1}}^{b_k} (2^{k})^{p(x)} w(x)^{p(x)}\,dx\\
\nonumber&\lesssim\sum\limits_{k\in \mbb{Z}} \int_{b_{k+1}}^{b_k} \left(\frac{1}{b_k}\int_0^{b_k} f_1(y)\,dy\right)^{p(x)} w(x)^{p(x)}\,dx\\
\nonumber&\lesssim\sum\limits_{k\in \mbb{Z}} \int_{b_{k+1}}^{b_k} \left(\int_0^{b_k} (f_1(y)\sigma^{-1}(y))\sigma(y)\,dy\right)^{p(x)} b_k^{-p(x)}w(x)^{p(x)}\,dx\\
\nonumber&=:J.
\end{align}

Since $f_1\sigma^{-1}> 1$ or $f_1\sigma^{-1}=0$, by
\eqref{eqn:modular} we have that 
\begin{multline*}
\int_0^{b_k} (f_1(y)\sigma^{-1}(y))^{\frac{p(y)}{p^-_k}}\sigma(y)\,dy
\leq \int_0^{b_k} (f_1(y)\sigma^{-1}(y))^{p(y)}\sigma(y)\,dy\\
=\int_0^{b_k} (f_1(y)w(y))^{p(y)}\,dy\leq 1.
\end{multline*}
Thus, for each $k\in\mbb{Z}$ and $x\in (b_{k+1}, b_k)$ we have
\begin{multline*}
\left(\int_0^{b_k} (f_1(y)\sigma^{-1}(y))\sigma(y)\,dy\right)^{p(x)}
\leq \left(\int_0^{b_k} (f_1(y)\sigma^{-1}(y))^{\frac{p(y)}{p^-_k}}
\sigma(y)\,dy\right)^{p(x)}\\
\leq \left(\int_0^{b_k} (f_1(y)\sigma^{-1}(y))^{\frac{p(y)}{p^-_k}}
\sigma(y)\,dy\right)^{p^-_k}.
\end{multline*}

Hence, by Jensen's inequality,
\begin{align*}
J
&\lesssim
\sum\limits_{k\in \mbb{Z}} 
\left(\int_0^{b_k} (f_1(y)\sigma^{-1}(y))^{\frac{p(y)}{p^-_k}}
\sigma(y)\,dy\right)^{p^-_k}\int_{b_{k+1}}^{b_k}  b_k^{-p(x)}w(x)^{p(x)}\,dx\\
&= \sum\limits_{k\in \mbb{Z}} \left(\frac{1}{\sigma(0,b_k)}
\int_0^{b_k} (f_1(y)\sigma^{-1}(y))^{\frac{p(y)}{p^-_k}}
\sigma(y)\,dy\right)^{p^-_k}\\
&\qquad\times
\int_{b_{k+1}}^{b_k} \sigma(0,b_k)^{p^-_k} b_k^{-p(x)}w(x)^{p(x)}\,dx\\
&\leq\sum\limits_{k\in \mbb{Z}} 
\left(\frac{1}{\sigma(0,b_k)}\int_0^{b_k} 
(f_1(y)\sigma^{-1}(y))^{\frac{p(y)}{p^-}}\sigma(y)\,dy\right)^{p^-}\\
&\qquad\times
\int_{b_{k+1}}^{b_k} \sigma(0,b_k)^{p^-_k} b_k^{-p(x)}w(x)^{p(x)}\,dx.
\end{align*}

To complete the proof we will estimate the last integral using the
$A_{p(\cdot),0}$ condition. We will show that
\begin{equation*}
\int_{b_{k+1}}^{b_k} \sigma(0,b_k)^{p^-_k} b_k^{-p(x)}w(x)^{p(x)}\,dx\leq C\sigma(0,b_k), \quad \text{for all } k\in \mbb{Z},
\end{equation*}
or, more generally,
\begin{equation}\label{A}
\int_{0}^{b} \sigma(0,b)^{p^-_{(0,b)}} b^{-p(x)}w(x)^{p(x)}\,dx\leq C\sigma(0,b), \quad \text{for all } b>0.
\end{equation}

From the $A_{p(\cdot),0}$ condition we know that
\[\left\|w\chi_{(0,b)}\frac{\|w^{-1}\chi_{(0,b)}\|_{p'(\cdot)}}{b}\right\|_{p(\cdot)}\leq C,\]
so by the definition of the norm,
\[\int_0^b
  \left(\frac{w(x)\|w^{-1}\chi_{(0,b)}\|_{p'(\cdot)}}{b}\right)^{p(x)}\,dx\leq
  C.\] 
Hence, it will suffice  to show that
\[\sigma(0,b)^{p^-_{(0,b)}}
\leq C \sigma(0,b)\|w^{-1}\chi_{(0,b)}\|_{p'(\cdot)}^{p(x)}\]
for every $x\in (0,b)$: that is,
\begin{equation}\label{5.5}
\frac{\sigma(0,b)^{p^-_{(0,b)}-1}}{\|w^{-1}\chi_{(0,b)}\|_{p'(\cdot)}^{p(x)}}\leq C,\quad x\in (0,b).
\end{equation}

The proof of (\ref{5.5}) when $\|w^{-1}\chi_{(0,b)}\|_{p'(\cdot)}>1$
is simple. By \cite[Corollary 2.23]{CUFbook}, we have
$\sigma(0,b)\leq
\|w^{-1}\chi_{(0,b)}\|_{p'(\cdot)}^{(p')^+_{(0,b)}}$. It is easy to
see that
\[(p')^+_{(0,b)}(p^-_{(0,b)}-1)=p^-_{(0,b)},\]
and since the exponent $p^-_{(0,b)}-p(x)$ is negative,
\[\frac{\sigma(0,b)^{p^-_{(0,b)}-1}}{\|w^{-1}\chi_{(0,b)}\|_{p'(\cdot)}^{p(x)}}\leq \|w^{-1}\chi_{(0,b)}\|_{p'(\cdot)}^{p^-_{(0,b)}-p(x)}\leq 1.\]

Now suppose that $\|w^{-1}\chi_{(0,b)}\|_{p'(\cdot)}\le1$.  Then,
$\sigma(0,b)\leq \|w^{-1}\chi_{(0,b)}\|_{p'(\cdot)}^{(p')^-_{(0,b)}}$.
Then by Lemma~\ref{Plog0norm} and 
\cite[Corollary 2.23]{CUFbook}, we have 
\begin{multline*}
\frac{\sigma(0,b)}{\| w^{-1}\chi_{(0,b)}\|_{p'(\cdot)}} 
\leq\frac{\| w^{-1}\chi_{(0,b)}\|_{p'(\cdot)}^{(p')^-_{(0,b)}}}
{\| w^{-1}\chi_{(0,b)}\|_{p'(\cdot)}}
= \| w^{-1}\chi_{(0,b)}\|_{p'(\cdot)}^{(p')^-_{(0,b)}-1}\\
\le C_0\| w^{-1}\chi_{(0,b)}\|_{p'(\cdot)}^{(p')^+_{(0,b)}-1}
\le  C \sigma(0,b)^{\frac{(p')^+_{(0,b)}-1}{(p')^+_{(0,b)}}}
=C\sigma(0,b)^{\frac{1}{p^-_{(0,b)}}}.
\end{multline*}
Consequently,
\begin{equation}\label{er1}
\sigma(0,b)^{p_{(0,b)}^--1}\le C\| w^{-1}\chi_{(0,b)}\|_{p'(\cdot)}^{p_{(0,b)}^-}.
\end{equation}

We now claim that 
\begin{equation}\label{er2}
\|w^{-1}\chi_{(0,b)}\|_{p'(\cdot)}^{p_{(0,b)}^--p(x)}\le C , \quad x\in (0,b).
\end{equation}
To prove this, we first estimate the exponent:
\begin{multline*}
p(x)-p_{(0,b)}^-
=\frac{p'(x)}{p'(x)-1}-\frac{(p')^+_{(0,b)}}{(p')^+_{(0,b)}-1}
=\frac{(p')^+_{(0,b)}-p'(x)}{(p'(x)-1)((p')^+_{(0,b)}-1)}\\
\leq \frac{(p')^+_{(0,b)}-(p')^-_{(0,b)}}{(p'(x)-1)((p')^+_{(0,b)}-1)}
\le \frac{(p')^+_{(0,b)}-(p')^-_{(0,b)}}{((p')^--1)^2}, 
\quad x\in (0,b).
\end{multline*}
Thus, for every $x\in (0,b)$,
\[
\| w^{-1}\chi_{(0,b)}\|_{p'(\cdot)}^{p_{(0,b)}^--p(x)}
=\left(\| w^{-1}\chi_{(0,b)}\|_{p'(\cdot)}^{-1} 
\right)^{p(x)-p_{(0,b)}^-}
\le\| w^{-1}\chi_{(0,b)}\|_{p'(\cdot)}
^{\frac{(p')^-_{(0,b)}-(p')^+_{(0,b)}}{((p')^--1)^2}},
\]
and by  Lemma \ref{Plog0norm} applied to $w^{-1}\in A_{p'(\cdot),0}$,
the right-hand term is bounded by a constant. This proves (\ref{er2}).
Together,  (\ref{er1}) and (\ref{er2}) immediately yield
\eqref{5.5}. 

Given \eqref{5.5}, we can now estimate as follows:
\[J\leq C\sum\limits_{k\in \mbb{Z}} 
\left(\frac{1}{\sigma(0,b_k)}\int_0^{b_k} (f_1(y)\sigma^{-1}(y))
^{\frac{p(y)}{p^-}}\sigma(y)\,dy\right)^{p^-}\sigma(0,b_k).\]
Since $\sigma\in A_{\infty,0}$ and $|[b_{k+1},b_k)|\geq |(0,b_k)|/2$,
by Lemma~\ref{Ainftymedidas} there exists $0<\beta<1$ such that
\[\sigma(b_{k+1},b_k) \geq \beta \sigma(0,b_k).\]
Define the weighted maximal operator\[N_\sigma g(x)=\sup_{b> x}\frac{1}{\sigma(0,b)} \int_0^b |g(y)|
\sigma(y)\,dy, \quad x>0.\]
From the $A_{p(\cdot),0}$ condition we have that $0<\sigma(0,b)<\infty$ for every $b>0$ (see the proof of Lemma \ref{Ainfty}). This fact together with \cite[Lemma 2.2 (2)]{DMRO1} implies that $N_\sigma$ is bounded on $L^{p^-}((0,\infty),d\sigma)$ since $p^->1$. Hence,
\begin{align*}
J&\leq C\sum\limits_{k\in \mbb{Z}} \left(\frac{1}{\sigma(0,b_k)}\int_0^{b_k} (f_1(y)\sigma^{-1}(y))^{\frac{p(y)}{p^-}}\sigma(y)\,dy\right)^{p^-}\sigma(b_{k+1},b_k)\\
&=C\sum\limits_{k\in \mbb{Z}} \int_{b_{k+1}}^{b_k}\left[N_\sigma\left((f_1\sigma^{-1})^{\frac{p(\cdot)}{p^-}}\right)(x)\right]^{p^-}\sigma(x)\,dx\\
&=C\int_0^\infty \left[N_\sigma\left((f_1\sigma^{-1})^{\frac{p(\cdot)}{p^-}}\right)(x)\right]^{p^-}\sigma(x)\,dx\\
&\leq C\int_0^\infty (f_1(x)\sigma^{-1}(x))^{p(x)}\sigma(x)\,dx\\
&= C\int_0^\infty f_1(x)^{p(x)}w(x)^{p(x)}\,dx\\
&\leq C.
\end{align*}

\textit{Estimate for $I_2$:} As we did for $f_1$, we can find a
non-increasing sequence $\{b_k\}_{k\in\mbb{Z}}$ such
that $\{x\in (0,\infty): Nf_2(x)>2^k\}=(0,b_k)$,
$\{x\in (0,\infty):2^k<Nf_2(x)\leq 2^{k+1}\}=[b_{k+1},b_k)$ and
\[2^k b_k=\int_0^{b_k} f_2(x) \,dx.\]
Then we can repeat the argument used in \eqref{decomp} to get
\[I_2\leq C\sum\limits_{k\in \mbb{Z}} 
\int_{b_{k+1}}^{b_k} \left(\frac{1}{b_k}
\int_0^{b_k} f_2(y)\,dy\right)^{p(x)} w(x)^{p(x)}\,dx.\]

By Lemmas~\ref{lemma:non-integrable} and~\ref{Ainfty}, $W$ and
$\sigma$ are not integrable over $(0,\infty)$. Thus, there exists
$c_0>0$ sufficiently large such that both $W(0,c_0)\geq 1$ and
$\sigma(0,c_0)\geq 1$. Let $I_0=(0,c_0)$;  we will split
the above sum into two pieces depending on the size of $b_k$:
\begin{align*}
I_2&\leq C\left(\sum\limits_{k:b_k\leq c_0}+\sum\limits_{k:b_k>c_0}\right) \int_{b_{k+1}}^{b_k} \left(\frac{1}{b_k}\int_0^{b_k} f_2(y)\,dy\right)^{p(x)} w(x)^{p(x)}\,dx\\
&=K_1+K_2.
\end{align*}
We will estimate each sum separately.

We first estimate $K_1$.  Since $f_2\sigma^{-1}\leq 1$, by inequality
\eqref{A} and the fact that $\sigma(0,b_k)\leq C\sigma(b_{k+1},b_k)$
 (since $\sigma \in A_{\infty ,0}$), we get
\begin{align*}
K_1
&\leq C\sum\limits_{k:b_k\leq c_0} 
\int_{b_{k+1}}^{b_k} \left(\frac{1}{b_k}
\int_0^{b_k} \sigma(y)\,dy\right)^{p(x)} w(x)^{p(x)}\,dx\\
&=C\sum\limits_{k:b_k\leq c_0} 
\int_{b_{k+1}}^{b_k} \sigma(0,b_k)^{p(x)-p^-_k}
\sigma(0,b_k)^{p^-_k}b_k^{-p(x)} w(x)^{p(x)}\,dx\\
&\leq C\sum\limits_{k:b_k\leq c_0} 
\left(1+\sigma(0,b_k)\right)^{p^+_k-p^-_k}
\int_{b_{k+1}}^{b_k}\sigma(0,b_k)^{p^-_k}b_k^{-p(x)} w(x)^{p(x)}\,dx\\
&\leq C\left(1+\sigma(I_0)\right)^{p^+-p^-}
\sum\limits_{k:b_k\leq c_0} \sigma(0,b_k)\\
&\leq C\left(1+\sigma(I_0)\right)^{p^+-p^-}
\sum\limits_{k:b_k\leq c_0} \sigma(b_{k+1},b_k)\\
&\leq C\left(1+\sigma(I_0)\right)^{p^+-p^-}\sigma(I_0) \\
& \leq C.
\end{align*}

We now estimate $K_2$.  Since $I_0\subset (0,b_k)$,
$\|w\chi_{(0,b_k)}\|_{p(\cdot)}^{-1}\leq
\|w\chi_{I_0}\|_{p(\cdot)}^{-1}$, and so by the $A_{p(\cdot),0}$
condition,
\begin{multline*}
\frac{1}{b_k}
\leq C\|w\chi_{(0,b_k)}\|_{p(\cdot)}^{-1}
\|w^{-1}\chi_{(0,b_k)}\|_{p'(\cdot)}^{-1} \\
\leq C\|w\chi_{I_0}\|_{p(\cdot)}^{-1}
\|w^{-1}\chi_{(0,b_k)}\|_{p'(\cdot)}^{-1}
\leq C \|w^{-1}\chi_{(0,b_k)}\|_{p'(\cdot)}^{-1}.
\end{multline*}
Hence, by H\"older's inequality and our assumptions on  $f$,
\[\frac{1}{Cb_k}\int_0^{b_k} f_2(y)\,dy\leq
  \|w^{-1}\chi_{(0,b_k)}\|_{p'(\cdot)}^{-1}
  \|f_2w\|_{p(\cdot)}\|w^{-1}\chi_{(0,b_k)}\|_{p'(\cdot)}\leq 1.\] 

Since $\pp\in LH_\infty(0,\infty)$, we can apply 
Lemma~\ref{changeexp}
with $d\mu(x)=w(x)^{p(x)}\,dx$, to the function
$g\equiv\frac{1}{Cb_k}\int_0^{b_k} f_2(y)\,dy\leq 1$ on
$G=[b_{k+1},b_k)$, to get
\begin{align*}
K_2
&\leq C\sum\limits_{k:b_k>c_0} \int_{b_{k+1}}^{b_k} 
\left(\frac{1}{Cb_k}\int_0^{b_k} f_2(y)\,dy\right)^{p(x)} w(x)^{p(x)}\,dx\\
&\leq C_t \sum\limits_{k:b_k>c_0} \int_{b_{k+1}}^{b_k} 
C^{-p_\infty}\left(\frac{1}{b_k}\int_0^{b_k}
  f_2(y)\,dy\right)^{p_\infty} 
w(x)^{p(x)}\,dx\\
&\qquad +\sum\limits_{k:b_k>c_0} 
\int_{b_{k+1}}^{b_k} \frac{w(x)^{p(x)}}{(e+x)^{tp^-}}\,dx\\
&\leq C_t \sum\limits_{k:b_k>c_0} \left(\frac{1}{b_k}
\int_0^{b_k} f_2(y)\,dy\right)^{p_\infty} W(b_{k+1},b_k)
+\int_{c_0}^{\infty} \frac{w(x)^{p(x)}}{(e+x)^{tp^-}}\,dx.
\end{align*}

Arguing as in the proof of Lemma \ref{Ainfty}, we can choose $t>1$
sufficiently large such that the second integral in the last line is
at most $1$.  To estimate the sum in the last line we start by rewriting it as
follows: 
\begin{align*}
\sum\limits_{k:b_k>c_0} 
&\left(\frac{1}{b_k}\int_0^{b_k} f_2(y)\,dy\right)^{p_\infty} 
W(b_{k+1},b_k)\\
&=\sum\limits_{k:b_k>c_0} \left(\frac{1}{\sigma(0,b_k)}
\int_0^{b_k} f_2(y)\sigma^{-1}(y)\sigma(y)\,dy\right)^{p_\infty} 
\left(\frac{\sigma(0,b_k)}{b_k}\right)^{p_\infty} W(b_{k+1},b_k)\\
&\leq C\sum\limits_{k:b_k>c_0} \left(\frac{1}{\sigma(0,b_k)}
\int_0^{b_k} f_2(y)\sigma^{-1}(y)\sigma(y)\,dy\right)^{p_\infty} 
\sigma(b_{k+1},b_k)\\
&\qquad
  \qquad\times\frac{\sigma(0,b_k)^{p_\infty-1}W(0,b_k)}{b_k^{p_\infty}},
\end{align*}
where we have used again that $\sigma\in A_{\infty,0}$. Since
$W(I_0),\sigma(I_0)\geq 1$, by \cite[Corollary 2.23]{CUFbook} we
have
$\|w\chi_{I_0}\|_{p(\cdot)}, \|w^{-1}\chi_{I_0}\|_{p'(\cdot)}\geq 1$,
so
$\|w\chi_{(0,b_k)}\|_{p(\cdot)},
\|w^{-1}\chi_{(0,b_k)}\|_{p'(\cdot)}\geq 1$ for every
$b_k>c_0$. Hence, we can apply Corollary \ref{normmodularpinfty}
twice and the $A_{p(\cdot),0}$ condition to get
\[\sigma(0,b_k)^{p_\infty-1}
=\sigma(0,b_k)^{\frac{p_\infty}{p'_\infty}}
\leq C\|w^{-1}\chi_{(0,b_k)}\|_{p'(\cdot)}^{p_\infty}
\leq C\frac{b_k^{p_\infty}}{\|w\chi_{(0,b_k)}\|_{p(\cdot)}^{p_\infty}}
\leq C \frac{b_k^{p_\infty}}{W(0,b_k)}.\]
Thus the final term is bounded.  

To estimate the sum, recall that 
since $p_\infty\geq p^->1$, $N_\sigma$ is bounded on
$L^{p_\infty}((0,\infty),d\sigma)$. Therefore, if we apply Lemma~\ref{changeexp}
with $d\mu(x)=\sigma(x)\,dx$ and $g=f_2\sigma^{-1}\leq 1$ on
$G=[b_{k+1},b_k)$, and use the boundedness of $N_\sigma$, we get
\begin{align*}
\sum\limits_{k:b_k>c_0} &
\left(\frac{1}{b_k}
\int_0^{b_k} f_2(y)\,dy\right)^{p_\infty} W(b_{k+1},b_k)\\
&\leq C\sum\limits_{k:b_k>c_0} 
\left(\frac{1}{\sigma(0,b_k)}
\int_0^{b_k} f_2(y)\sigma^{-1}(y)\sigma(y)\,dy\right)^{p_\infty} 
\sigma(b_{k+1},b_k)\\
&\leq C\sum\limits_{k:b_k>c_0} 
\int_{b_{k+1}}^{b_k}N_\sigma(f_2\sigma^{-1})(x)^{p_\infty} 
\sigma(x)\,dx\\
&\leq C\int_0^\infty N_\sigma(f_2\sigma^{-1})(x)^{p_\infty} 
\sigma(x)\,dx\\
&\leq C\int_0^\infty (f_2(x)\sigma^{-1}(x))^{p_\infty} 
\sigma(x)\,dx\\
&\leq C_t C\int_0^\infty (f_2(x)\sigma^{-1}(x))^{p(x)} 
\sigma(x)\,dx
+\int_0^\infty \frac{\sigma(x)}{(e+x)^{tp^-}}\,dx\\
&\leq C_t+1.
\end{align*}
In the second to last inequality we again used Lemma~\ref{changeexp},
exchanging the roles of $\pp$ and $p_\infty$ and replacing $w$ by
$\sigma$.  In the final inequality we used the fact that
\[ \int_0^\infty (f_2(x)\sigma^{-1}(x))^{p(x)} 
\sigma(x)\,dx = \int_0^\infty f_2(x)^{p(x)}w(x)^{p(x)}\,dx\leq 1.\]
To estimate the final integral, we argued as we did in the proof of
Lemma~\ref{Ainfty} with $\sigma$ instead of $W$, to show that
we could choose $t$ big enough so that this term is smaller
than $1$.  This completes the proof.
\end{proof}

\section{Proofs of Theorems \ref{strongtypeS-C} 
and \ref{Calpha}}  
\label{secC-Calpha}

We will prove Theorem~\ref{strongtypeS-C} in two steps. First, we will
prove it when $\lambda=1$. Then we will give two lemmas that let us
prove it for every $0<\lambda<1$.

\begin{proof}[Proof of Theorem \ref{strongtypeS-C} for $\lambda=1$] 
  As we have remarked in the introduction, $\mc{C}f\sim Sf$; hence, it
  will suffice to prove that \textit{(i), (iii)} and \textit{(v)} are
  equivalent. Clearly, \textit{(i)} implies \textit{(iii)}.
	Similarly, \textit{(iii)}$\Rightarrow$\textit{(v)} is immediate:
        since $Nf \lesssim \mc C f$, if $\mc C$ is of weak-type,
        then $N$ is weak-type, and by Theorem \ref{strongtypeN}, we
        get that $w\in A_{p(\cdot),0}$.

        Finally, we will show that
        \textit{(v)}$\Rightarrow$\textit{(i)}. If
        $w\in A_{p(\cdot),0}$, then $w^{-1} \in A_{\cpp,0}$, and so by Theorem \ref{strongtypeN},
        $N$ is bounded on
        $L^{p(\cdot)}_w(0,\infty)$ and
        $L^{p'(\cdot)}_{w^{-1}}(0,\infty)$. Since $Hf\leq Nf$ for non-negative
        $f$, $H$ is bounded on $L^{p(\cdot)}_w(0,\infty)$ and
        $L^{p'(\cdot)}_{w^{-1}}(0,\infty)$. Then by duality we
         also have that $H^*$ is bounded on $L^{p(\cdot)}_w(0,\infty)$.
Therefore, 
\[ \|(\mc{C}f)w\|_{p(\cdot)}
\leq \|(Hf)w\|_{p(\cdot)}+\|(H^*f)w\|_{p(\cdot)}
\leq K\|fw\|_{p(\cdot)}. \]
This completes the proof when $\lambda=1$.
\end{proof}

\medskip

In order to prove Theorem \ref{strongtypeS-C}  when $\lambda\in(0,1)$,
we need two  lemmas. The first lets us relate the $A_{p(\cdot),
  q(\cdot),0}$ to the $A_{p(\cdot),0}$ condition.  This result is
analogous to a property of the $A_{p,q}$ weights in~\cite{MW} and the
$A_{\pp,\qq}$ weights proved in~\cite{BDP}.

\begin{lem}\label{clases}
  Given $\pp\in \Pp$ and $\lambda>0$, define $\qq$ as in the statement
  of Theorem~\ref{strongtypeS-C}. Then $w\in A_{p(\cdot),q(\cdot),0}$
  if and only if $w^{1/\lambda}\in A_{\lambda q(\cdot),0}$.
\end{lem}

\begin{proof}
  The proof is essentially the same as the proof of the corresponding
  result for the $A_{\pp,\qq}$ and $A_{\pp}$ classes.  More precisely,
  it is enough to consider intervals of the form $\{(0,b) : b>0 \}$,
  $n=1$, and $\alpha=1-\lambda$ in the proof of \cite[Lemma 4.1~(i)]{BDP}.
\end{proof}

The second lemma is a Hedberg-type inequality (see \cite[Eq.(5)]{Hed}) which
lets us control $S_\lambda$ with $S=S_1$. 

\begin{lem}\label{Hedberg}
  Given $\pp\in \Pp(0,\infty)$ and $\lambda>0$, define $\qq$ as in the statement
  of Theorem~\ref{strongtypeS-C}.  Let $w$ be a weight and let $f$ be
  a non-negative function in $L^{p(\cdot)}(0,\infty)$. Then for every
  $x\in (0,\infty)$,
\[S_\lambda\left(\frac{f}{w}\right)(x)\leq \left[S
    \left(g^{1/\lambda}\right)(x)\right]^\lambda \left(\int_0^\infty
    f(y)^{p(y)}\,dy\right)^{1-\lambda},\]
where $g(y)=f(y)^{p(y)/q(y)}w^{-1}(y)$.
\end{lem}

\begin{proof}
  We adapt the argument given in \cite{GPS10} for the fractional
  maximal operator with weights (see also \cite{BDP, GPS12}).  From
  the definition of $g$ and the relation between $\pp$ and $\qq$ we get
\[f(y)w^{-1}(y)=g(y)f(y)^{1-p(y)/q(y)}=g(y)f(y)^{(1-\lambda)p(y)}.\]
Thus, if we apply H\"older's inequality with $1/\lambda>1$ and
$(1/\lambda)'=1/(1-\lambda)$, we get
\begin{align*}
S_\lambda\left(\frac{f}{w}\right)(x)
&=\int_0^\infty \frac{g(y)}{(x+y)^\lambda}f(y)^{(1-\lambda)p(y)} \,dy\\
&\leq \left(\int_0^\infty \frac{g(y)^{1/\lambda}}{x+y}
  \,dy\right)^\lambda 
\left(\int_0^\infty f(y)^{p(y)} \,dy\right)^{1-\lambda}\\
&=\left[S \left(g^{1/\lambda}\right)(x)\right]^\lambda 
\left(\int_0^\infty f(y)^{p(y)}\,dy\right)^{1-\lambda}.\qedhere
\end{align*}
\end{proof}

\begin{proof}[Proof of Theorem \ref{strongtypeS-C} for $\lambda\in
  (0,1)$] 
  As in the case $\lambda=1$, since $S_\lambda\sim \mc{C}_\lambda$ and
  the strong-type implies the weak-type, it is enough to prove that 
  \textit{(iii)}$\Rightarrow$\textit{(v)} and
  \textit{(v)}$\Rightarrow$\textit{(i)}.
	
  To prove \textit{(iii)}$\Rightarrow$\textit{(v)} we argue as in the
  proof of necessity in Theorem~\ref{strongtypeN}.  Fix $b>0$; then
  there exists a non-negative function   $g\in L^{p(\cdot)}(0,\infty)$
  such that  $\|g\|_{p(\cdot)}\leq 1$ and
\[\|w^{-1}\chi_{(0,b)}\|_{p'(\cdot)}\sim \int_0^b w^{-1}(y)g(y)\,dy.\]
Without loss of generality we may assume that
$\|w^{-1}\chi_{(0,b)}\|_{p'(\cdot)}>0$.  Define
$f=\chi_{(0,b)}w^{-1}g$; then $f\in L^{p(\cdot)}_w(0,\infty)$ with
$\|fw\|_{p(\cdot)}= \|\chi_{(0,b)}g\|_{p(\cdot)}\leq 1$. If
$x\in(0,b)$, 
\begin{multline*}
S_\lambda f(x)
=\int_0^b \frac{w^{-1}(y)g(y)}{(x+y)^\lambda}\,dy
> \frac{1}{(2b)^\lambda}\int_0^b w^{-1}(y)g(y)\,dy\\
\sim \frac{\|w^{-1}\chi_{(0,b)}\|_{p'(\cdot)}}{b^\lambda}:=\mu>0.
\end{multline*}
Hence,  $(0,b)\subset \{x\in (0,\infty): S_\lambda f(x)>\mu\}$, so
by the weak-type inequality we have that
\begin{equation*}
\frac{\|w^{-1}\chi_{(0,b)}\|_{p'(\cdot)}}{b^\lambda}\|w\chi_{(0,b)}\|_{q(\cdot)}
\lesssim  \mu \|w\chi_{\{x\in (0,\infty): S_\lambda
  f(x)>\mu\}}\|_{q(\cdot)}
\leq  K\|fw\|_{p(\cdot)}\leq K,
\end{equation*}
or, equivalently,
\[\|w\chi_{(0,b)}\|_{q(\cdot)}\|w^{-1}\chi_{(0,b)}\|_{p'(\cdot)} \leq
  C b^\lambda.\]
Since $b>0$ is arbitrary, we get that $w \in A_{p(\cdot), q(\cdot),0}$.

To prove \textit{(v)}$\Rightarrow$\textit{(i)}, fix
$w\in A_{p(\cdot),q(\cdot),0}$.  To show that this implies 
$\|\left(S_\lambda f\right) w\|_{q(\cdot)}\leq C\|fw\|_{p(\cdot)}$
for every $f\in L^{p(\cdot)}_w(0,\infty)$,  we will prove an
equivalent inequality:  for every $f\in L^{p(\cdot)}(0,\infty)$,
\[\|\left(S_\lambda (f/w)\right) w\|_{q(\cdot)}\leq
  C\|f\|_{p(\cdot)}. \]

Without loss of generality, we may assume $\|f\|_{p(\cdot)}=1$, so
that $\int_0^\infty f(y)^{p(y)}\,dy=1$.  We will show that
\[\|\left(S_\lambda (f/w)\right) w\|_{q(\cdot)}\leq C. \]
By Lemma \ref{Hedberg}, we have
\[S_\lambda\left(\frac{f}{w}\right)(x)
\leq \left[S \left(g^{1/\lambda}\right)(x)\right]^\lambda\]
with $g(y)=f(y)^{p(y)/q(y)}w^{-1}(y)$. 
Therefore,
\[\|\left(S_\lambda (f/w)\right) w\|_{q(\cdot)}
\leq \|S (g^{1/\lambda})^\lambda w \|_{q(\cdot)}
=\|S (g^{1/\lambda}) w^{1/\lambda} \|_{\lambda q(\cdot)}^\lambda.\]

Observe that
\[\int_0^\infty \left(g(y)^{1/\lambda}
    w(y)^{1/\lambda}\right)^{\lambda q(y)} \,dy
=\int_0^\infty \left(f(y)^{p(y)/q(y)}\right)^{q(y)} \,dy=1,\]
so $g^{1/\lambda}\in L^{\lambda q(\cdot)}_{w^{1/\lambda}}(0,\infty)$
with
$\|g^{1/\lambda}w^{1/\lambda}\|_{\lambda
  q(\cdot)}=\|gw\|_{q(\cdot)}=1$.  Further, we have that
$\qq\in LH_0(0,\infty)\cap LH_\infty(0,\infty)$ since $\pp$ belongs to
both classes and $p^+<1/(1-\lambda)$.  By Lemma \ref{clases},
$w^{1/\lambda}\in A_{\lambda q(\cdot),0}$ so by the case $\lambda=1$
proved above, $S$ is bounded on $L^{\lambda
  q(\cdot)}_{w^{1/\lambda}}$. Therefore, we have that
\[\|\left(S_\lambda (f/w)\right) w\|_{q(\cdot)}
\leq C^\lambda \|g^{1/\lambda}w^{1/\lambda}\|_{\lambda q(\cdot)}\leq C.\]
This completes the proof.
\end{proof}

\begin{proof}[Proof of Theorem \ref{Calpha}] 
  Since $\mc{C}^\alpha f\leq \mc{C}f$ for non-negative functions $f$,
  by Theorem \ref{strongtypeS-C} and the fact that the
  strong-type inequality  implies the weak-type, it suffices to show that
  \textit{\ref{weakCalpha}} implies \textit{\ref{wAp0}}.
	
  We argue as we did for the proof of necessity above.  Fix $b>0$;
  then by duality there exists a non-negative function $g$ such that
  $\|g\|_{p(\cdot)}\leq 1$  and
\[\|w^{-1}\chi_{(0,b)}\|_{p'(\cdot)}\sim \int_0^b w^{-1}(x)g(x)
  \,dx.\]
Again we may assume $\|w^{-1}\chi_{(0,b)}\|_{p'(\cdot)}>0$. Let
$f=\chi_{(0,b)}gw^{-1}$; then for  $t\in (2b,3b)$, 
\begin{multline*}
\mc{C}^\alpha f(t)
\geq \frac{1}{(3b)^{\alpha+1}}\int_0^b (t-x)^{\alpha} w^{-1}(x)g(x) \,dx\\
\geq \frac{1}{3^{\alpha+1}b}\int_0^b w^{-1}(x)g(x) \,dx
\sim \frac{\|w^{-1}\chi_{(0,b)}\|_{p'(\cdot)}}{b}:=\mu>0.
\end{multline*}
Therefore,
\[(2b,3b)\subset \{t\in (0,\infty): \mc{C}^\alpha f(t)>\mu\}.\]
By the weak-type inequality and the choice of $g$, we get that
\begin{equation}\label{2b3b}
\|\chi_{(2b,3b)}w\|_{p(\cdot)}\frac{\|w^{-1}\chi_{(0,b)}\|_{p'(\cdot)}}{b}
\lesssim K\|w\chi_{(0,b)}w^{-1}g\|_{p(\cdot)}\leq K.
\end{equation}

On the other hand, if we let $f=\chi_{(2b,3b)}$, then it follows
from~\eqref{2b3b} that $f \in L^{p(\cdot)}_w(0,\infty)$. 
Thus, if we take $t\in(0,b)$, we have that $\mc{C}^\alpha f(t)\geq
3^{-\alpha-1}$, and so
\[(0,b)\subset \{t\in (0,\infty): \mc{C}^\alpha
  f(t)>3^{-\alpha-1}\}.\]
Therefore,  again by the weak-type inequality,
we have that
\[3^{-\alpha-1}\|w\chi_{(0,b)}\|_{p(\cdot)}\leq
  K\|w\chi_{(2b,3b)}\|_{p(\cdot)}.\]
If we combine this inequality with \eqref{2b3b}, we see that
$w\in A_{p(\cdot),0}$.
\end{proof}

\section{Proofs of Theorems~\ref{thm:n-dim}, \ref{Hilbertineq} and
  \ref{SoriaWeiss}}  
\label{app}

\begin{proof}[Proof of Theorem~\ref{thm:n-dim}]
  The proof of these results in $\R^n$, $n>1$, is essentially the same
  as the proof of the one-dimensional results on $(0,\infty)$.  In the
  definition of $A_{\infty,0}$, we replace $b$ in the denominator by
  $b^n$ or by the volume of the ball $B(0,b)$. The proof of
  Lemma~\ref{lemma:non-integrable} relies on results
  from~\cite{DMRO2}, but these are for abstract bases over measure
  spaces and so hold in higher dimensions.  In the proofs of the
  lemmas in Section~\ref{lemmas} and in the proofs of
  Theorem~\ref{strongtypeN} and of Theorem~\ref{strongtypeS-C} for
  $\lambda=1$, we replace $(0,\infty)$ by $\R^n$, the intervals
  $(0,b)$ by the balls $B(0,b)$ and intervals of the form $(a,b)$ by
  the annuli $\{ x\in \R^n : a<|x|<b \}$.  

The proofs then go through
  exactly the same as in the one-dimensional case. We used the fact that the weighted maximal operator
 $N_\sigma$ is bounded on $L^p((0,\infty),d\sigma)$ for any $1<p<\infty$, proved in \cite{DMRO1}. Now, we need to show that the corresponding operator on $\mathbb R^n$, given by
\[N_\sigma f(x) 
= \sup_{b>|x|} \frac{1}{\sigma(B(0,b))}
\int_{B(0,b)} |f(y)|\sigma(y)\,dy\]
is bounded on $L^p(\mathbb R^n, d\sigma)$ for $1<p<\infty$.  We include the proof below, which was sketched in \cite[pp. 559-560]{Du}. First, notice that the $A_{p(\cdot),0}$ condition will guarantee $0<\sigma(B(0,b))<\infty$. Then, we can show that $N_\sigma$ satisfies a weak $(1,1)$ inequality, as in the one-dimensional case (see \cite[Lemma 2.2]{DMRO1}). Suppose $f$ is
a bounded function of  compact support. Then, we have that given any $\mu>0$, there exists $b=b(\mu)>0$ such that
\[ \{ x \in \R^n : N_\sigma f(x) > \mu \} = B(0,b), \]
and 
\[ \mu = \frac{1}{\sigma(B(0,b))}
\int_{B(0,b)} |f(y)|\sigma(y)\,dy. \]
But then we immediately get the weak $(1,1)$ inequality:
\[ \sigma(\{ x \in \R^n : N_\sigma f(x) > \mu \})
= \sigma(B(0,b))
\leq \frac{1}{\mu}\int_{\R^n} |f(y)|\sigma(y)\,dy.\]
That $ N_\sigma$ is bounded on $L^p(\mathbb R^n, d\sigma)$ for $p>1$ now follows from
Marcinkiewicz interpolation.
\end{proof}

\begin{proof}[Proof of Theorem \ref{Hilbertineq}]
By Theorem~\ref{thm:n-dim}, $w\in A_{p(\cdot),0}$ is equivalent to
\[\|(Sf)w\|_{p(\cdot)}\leq C \|fw\|_{p(\cdot)}.\]
By duality, this inequality can be rewritten as
\[\sup\limits_{\|gw^{-1}\|_{p'(\cdot)}\leq 1} 
\int_{\mbb{R}^n} Sf(x) g(x) \,dx \leq C \|fw\|_{p(\cdot)},\]
which in turn is equivalent to
\[\sup\limits_{\substack{g\in L_{w^{-1}}^{p'(\cdot)}(\mbb{R}^n)\\
g      \geq 0}}
\int_{\mbb{R}^n} \left(\int_{\mbb R^n}
  \frac{f(x)}{|x|^n+|y|^n}\,dy\right) 
\frac{g(x)}{\|gw^{-1}\|_{p'(\cdot)}} \,dx 
\leq C \|fw\|_{p(\cdot)}.\]
This in turn is equivalent to the desired inequality \eqref{Hilbert}.
\end{proof}

\begin{proof}[Proof of Theorem \ref{SoriaWeiss}]
  For each $k\in \mbb{Z}$, define the annuli
  $I_k=\{x\in \mbb{R}^n: 2^{k-1}\leq|x|< 2^k\}$ and
  $I_k^*=\{x\in \mbb{R}^n: 2^{k-2}\leq|x|< 2^{k+1}\}$. Note that the
  $I_k^*$ have bounded overlap.  
Given $f\in L^{p(\cdot)}_w(\mbb{R}^n)$, let $f_{k,0}=f\chi_{I_k^*}$ and
$f_{k,1}=f-f_{k,0}$.  Then we have that
\begin{multline*}
T^*f(x)
=\sum\limits_{k\in \mbb{Z}} T^*f(x)\chi_{I_k}(x)\\ 
\leq \sum\limits_{k\in \mbb{Z}} T^*f_{k,0}(x)\chi_{I_k}(x)+\sum\limits_{k\in \mbb{Z}} T^*f_{k,1}(x)\chi_{I_k}(x)
:=T_0^*f(x)+T_1^*f(x).
\end{multline*}

For the operator $T_0^*$, we will use duality, \eqref{constannuli},
the boundedness of $T^*$ and \eqref{classG} to get
\begin{align*}
\|wT^*_0 f\|_{p(\cdot)}
&\leq C \sup\limits_{\|g\|_{p'(\cdot)}\leq 1}
\int_{\mbb{R}^n} T_0^* f(x) g(x) w(x) \,dx\\
&\leq C\sup\limits_{\|g\|_{p'(\cdot)}\leq 1}
\sum\limits_{k\in\mbb{Z}}\int_{I_k}|T^*f_{k,0}(x)\|g(x)|w(x)\,dx\\
&\leq C\sup\limits_{\|g\|_{p'(\cdot)}\leq 1}
\sum\limits_{k\in\mbb{Z}} 
\sup\limits_{I_k^*}w(x)\int_{I_k}|T^*f_{k,0}(x)\|g(x)|\,dx\\
&\leq C\sup\limits_{\|g\|_{p'(\cdot)}\leq 1}
\sum\limits_{k\in\mbb{Z}} 
\sup\limits_{I_k^*}w(x) \|T^*f_{k,0}\|_{p(\cdot)}\|g\chi_{I_k}\|_{p'(\cdot)}\\
&\leq C\sup\limits_{\|g\|_{p'(\cdot)}\leq 1}
\sum\limits_{k\in\mbb{Z}} 
\inf\limits_{I_k^*}w(x) \|f\chi_{I_k^*}\|_{p(\cdot)}\|g\chi_{I_k^*}\|_{p'(\cdot)}\\
&\leq C\sup\limits_{\|g\|_{p'(\cdot)}\leq 1}
\sum\limits_{k\in\mbb{Z}} 
\|wf\chi_{I_k^*}\|_{p(\cdot)}\|g\chi_{I_k^*}\|_{p'(\cdot)}\\
&\leq C\sup\limits_{\|g\|_{p'(\cdot)}\leq 1}
\|fw\|_{p(\cdot)}\|g\|_{p'(\cdot)}\\
&\leq  C\|fw\|_{p(\cdot)}.
\end{align*}

In order to estimate $T_1^*$, first note that for $x\in I_k$ and
$y\in (I_k^*)^c$, $|x-y|\sim |x|+|y|$.  Then by  \eqref{boundKj} we
have the pointwise estimate
\[T_1^*f(x)\leq C_0\sum\limits_{k\in\mbb{Z}}
\left(\int_{(I_k^*)^c}\frac{f(y)}{|x-y|^n}\,dy\right)\chi_{I_k}(x)
\leq C_0\int_{\mbb{R}^n}\frac{|f(y)|}{|x|^n+|y|^n}\,dy
= C_0 Sf(x).\]
Since $w\in A_{p(\cdot),0}$, the desired bound follows from
Theorem~\ref{thm:n-dim}.  
\end{proof}

\bibliographystyle{acm}
\bibliography{ref}

\def\ocirc#1{\ifmmode\setbox0=\hbox{$#1$}\dimen0=\ht0 \advance\dimen0
  by1pt\rlap{\hbox to\wd0{\hss\raise\dimen0
  \hbox{\hskip.2em$\scriptscriptstyle\circ$}\hss}}#1\else {\accent"17 #1}\fi}
\begin{thebibliography}{10}

\bibitem{A}
{\sc Andersen, K.~F.}
\newblock Weighted inequalities for the {S}tieltjes transformation and
  {H}ilbert's double series.
\newblock {\em Proc. Roy. Soc. Edinburgh Sect. A 86}, 1-2 (1980), 75--84.

\bibitem{AM}
{\sc Ari\~{n}o, M.~A., and Muckenhoupt, B.}
\newblock Maximal functions on classical {L}orentz spaces and {H}ardy's
  inequality with weights for nonincreasing functions.
\newblock {\em Trans. Amer. Math. Soc. 320}, 2 (1990), 727--735.

\bibitem{B}
{\sc Bandaliev, R.~A.}
\newblock The boundedness of certain sublinear operator in the weighted
  variable {L}ebesgue spaces.
\newblock {\em Czechoslovak Math. J. 60(135)}, 2 (2010), 327--337.

\bibitem{B2}
{\sc Bandaliev, R.~A.}
\newblock Corrections to the paper ``{T}he boundedness of certain sublinear
  operator in the weighted variable {L}ebesgue spaces'' [{MR}2657952].
\newblock {\em Czechoslovak Math. J. 63(138)}, 4 (2013), 1149--1152.

\bibitem{BMR}
{\sc Bastero, J., Milman, M., and Ruiz, F.~J.}
\newblock On the connection between weighted norm inequalities, commutators and
  real interpolation.
\newblock {\em Mem. Amer. Math. Soc. 154}, 731 (2001), viii+80.

\bibitem{Ber}
{\sc Berezhnoi, E.~I.}
\newblock Two-weighted estimations for the {H}ardy-{L}ittlewood maximal
  function in ideal {B}anach spaces.
\newblock {\em Proc. Amer. Math. Soc. 127}, 1 (1999), 79--87.

\bibitem{BDP}
{\sc Bernardis, A., Dalmasso, E., and Pradolini, G.}
\newblock Generalized maximal functions and related operators on weighted
  {M}usielak-{O}rlicz spaces.
\newblock {\em Ann. Acad. Sci. Fenn. Math. 39}, 1 (2014), 23--50.

\bibitem{CruzUribe:2016ji}
{\sc Cruz-Uribe, D.}
\newblock {Two weight inequalities for fractional integral operators and
  commutators}.
\newblock In {\em VI International Course of Mathematical Analysis in
  Andalusia\/} (2016), F.~J. Martin-Reyes, Ed., World Scientific, pp.~25--85.

\bibitem{CUDH}
{\sc Cruz-Uribe, D., Diening, L., and H\"{a}st\"{o}, P.}
\newblock The maximal operator on weighted variable {L}ebesgue spaces.
\newblock {\em Fract. Calc. Appl. Anal. 14}, 3 (2011), 361--374.

\bibitem{CUFbook}
{\sc Cruz-Uribe, D., and Fiorenza, A.}
\newblock {\em Variable {L}ebesgue spaces}.
\newblock Applied and Numerical Harmonic Analysis. Birkh\"auser/Springer,
  Heidelberg, 2013.
\newblock Foundations and harmonic analysis.

\bibitem{CUFN}
{\sc Cruz-Uribe, D., Fiorenza, A., and Neugebauer, C.~J.}
\newblock Weighted norm inequalities for the maximal operator on variable
  {L}ebesgue spaces.
\newblock {\em J. Math. Anal. Appl. 394}, 2 (2012), 744--760.

\bibitem{CUM}
{\sc Cruz-Uribe, D., and Mamedov, F.~I.}
\newblock On a general weighted {H}ardy type inequality in the variable
  exponent {L}ebesgue spaces.
\newblock {\em Rev. Mat. Complut. 25}, 2 (2012), 335--367.

\bibitem{CUW}
{\sc Cruz-Uribe, D., and Wang, L.-A.}
\newblock Extrapolation and weighted norm inequalities in the variable
  {L}ebesgue spaces.
\newblock {\em Trans. Amer. Math. Soc. 369}, 2 (2017), 1205--1235.

\bibitem{Diening}
{\sc Diening, L.}
\newblock Maximal function on generalized {L}ebesgue spaces {$L\sp
  {p(\cdot)}$}.
\newblock {\em Math. Inequal. Appl. 7}, 2 (2004), 245--253.

\bibitem{DHHR}
{\sc Diening, L., Harjulehto, P., H{\"a}st{\"o}, P., and
  R{\ocirc{u}}{\v{z}}i{\v{c}}ka, M.}
\newblock {\em Lebesgue and {S}obolev spaces with variable exponents},
  vol.~2017 of {\em Lecture Notes in Mathematics}.
\newblock Springer, Heidelberg, 2011.

\bibitem{DS}
{\sc Diening, L., and Samko, S.}
\newblock Hardy inequality in variable exponent {L}ebesgue spaces.
\newblock {\em Fract. Calc. Appl. Anal. 10}, 1 (2007), 1--18.

\bibitem{Du}
{\sc Duoandikoetxea, J.}
\newblock Fractional integrals on radial functions with applications to
  weighted inequalities.
\newblock {\em Ann. Mat. Pura Appl. (4) 192}, 4 (2013), 553--568.

\bibitem{DMRO1}
{\sc Duoandikoetxea, J., Mart{\'{\i}}n-Reyes, F.~J., and Ombrosi, S.}
\newblock Calder\'on weights as {M}uckenhoupt weights.
\newblock {\em Indiana Univ. Math. J. 62}, 3 (2013), 891--910.

\bibitem{DMRO2}
{\sc Duoandikoetxea, J., Mart\'{i}n-Reyes, F.~J., and Ombrosi, S.}
\newblock On the {$A_\infty$} conditions for general bases.
\newblock {\em Math. Z. 282}, 3-4 (2016), 955--972.

\bibitem{GKP13}
{\sc Gogatishvili, A., Kufner, A., and Persson, L.-E.}
\newblock The weighted {S}tieltjes inequality and applications.
\newblock {\em Math. Nachr. 286}, 7 (2013), 659--668.

\bibitem{GKPW}
{\sc Gogatishvili, A., Kufner, A., Persson, L.-E., and Wedestig, A.}
\newblock An equivalence theorem for integral conditions related to {H}ardy's
  inequality.
\newblock {\em Real Anal. Exchange 29}, 2 (2003/04), 867--880.

\bibitem{GPSW}
{\sc Gogatishvili, A., Persson, L.-E., Stepanov, V.~D., and Wall, P.}
\newblock Some scales of equivalent conditions to characterize the {S}tieltjes
  inequality: the case {$q < p$}.
\newblock {\em Math. Nachr. 287}, 2-3 (2014), 242--253.

\bibitem{GPS10}
{\sc Gorosito, O., Pradolini, G., and Salinas, O.}
\newblock Boundedness of fractional operators in weighted variable exponent
  spaces with non doubling measures.
\newblock {\em Czechoslovak Math. J. 60(135)}, 4 (2010), 1007--1023.

\bibitem{GPS12}
{\sc Gorosito, O., Pradolini, G., and Salinas, O.}
\newblock Boundedness of the fractional maximal operator on variable exponent
  {L}ebesgue spaces: a short proof.
\newblock {\em Rev. Un. Mat. Argentina 53}, 1 (2012), 25--27.

\bibitem{H}
{\sc Hardy, G.~H.}
\newblock Note on a theorem of {H}ilbert concerning series of positive terms.
\newblock {\em Proc. Lond. Math. Soc. (2)}, 23 (1925).
\newblock Records of Proc. XLV-XLVI.

\bibitem{HLP}
{\sc Hardy, G.~H., Littlewood, J.~E., and P\'{o}lya, G.}
\newblock {\em Inequalities}.
\newblock Cambridge Mathematical Library. Cambridge University Press,
  Cambridge, 1988.
\newblock Reprint of the 1952 edition.

\bibitem{HM}
{\sc Harman, A., and Mamedov, F.~I.}
\newblock On boundedness of weighted {H}ardy operator in {$L^{p(\cdot)}$} and
  regularity condition.
\newblock {\em J. Inequal. Appl.\/} (2010), Art. ID 837951, 14.

\bibitem{Hed}
{\sc Hedberg, L.}
\newblock On certain convolution inequalities.
\newblock {\em Proc. Amer. Math. Soc. 36\/} (1972), 505--510.

\bibitem{Hy}
{\sc Hyt\"{o}nen, T.}
\newblock The {$A_2$} theorem: remarks and complements.
\newblock In {\em Harmonic analysis and partial differential equations},
  vol.~612 of {\em Contemp. Math.} Amer. Math. Soc., Providence, RI, 2014,
  pp.~91--106.

\bibitem{KR}
{\sc Kov{\'a}{\v{c}}ik, O., and R{\'a}kosn{\'{\i}}k, J.}
\newblock On spaces {$L\sp {p(x)}$} and {$W\sp {k,p(x)}$}.
\newblock {\em Czechoslovak Math. J. 41(116)}, 4 (1991), 592--618.

\bibitem{L}
{\sc Lerner, A.~K.}
\newblock On a dual property of the maximal operator on weighted variable
  {$L^p$} spaces.
\newblock In {\em Functional analysis, harmonic analysis, and image processing:
  a collection of papers in honor of {B}j\"{o}rn {J}awerth}, vol.~693 of {\em
  Contemp. Math.} Amer. Math. Soc., Providence, RI, 2017, pp.~283--300.

\bibitem{MH09}
{\sc Mamedov, F.~I., and Harman, A.}
\newblock On a weighted inequality of {H}ardy type in spaces {$L^{p(\cdot)}$}.
\newblock {\em J. Math. Anal. Appl. 353}, 2 (2009), 521--530.

\bibitem{MH10}
{\sc Mamedov, F.~I., and Harman, A.}
\newblock On a {H}ardy type general weighted inequality in spaces
  {$L^{p(\cdot)}$}.
\newblock {\em Integral Equations Operator Theory 66}, 4 (2010), 565--592.

\bibitem{MZ}
{\sc Mamedov, F.~I., and Zeren, Y.}
\newblock On equivalent conditions for the general weighted {H}ardy type
  inequality in space {$L^{p(\cdot)}$}.
\newblock {\em Z. Anal. Anwend. 31}, 1 (2012), 55--74.

\bibitem{MCMO}
{\sc Mashiyev, R.~A., \c{C}eki\c{c}, B., Mamedov, F.~I., and Ogras, S.}
\newblock Hardy's inequality in power-type weighted {$L^{p(\cdot)}(0,\infty)$}
  spaces.
\newblock {\em J. Math. Anal. Appl. 334}, 1 (2007), 289--298.

\bibitem{M1}
{\sc Muckenhoupt, B.}
\newblock Hardy's inequality with weights.
\newblock {\em Studia Math. 44\/} (1972), 31--38.
\newblock Collection of articles honoring the completion by Antoni Zygmund of
  50 years of scientific activity, I.

\bibitem{M2}
{\sc Muckenhoupt, B.}
\newblock Weighted norm inequalities for the {H}ardy maximal function.
\newblock {\em Trans. Amer. Math. Soc. 165\/} (1972), 207--226.

\bibitem{MW}
{\sc Muckenhoupt, B., and Wheeden, R.}
\newblock Weighted norm inequalities for fractional integrals.
\newblock {\em Trans. Amer. Math. Soc. 192\/} (1974), 261--274.

\bibitem{O}
{\sc Orlicz, W.}
\newblock {\"Uber konjugierte Exponentenfolgen.}
\newblock {\em Studia Math. 3\/} (1931), 200--211.

\bibitem{Sch}
{\sc Schur, J.}
\newblock Bemerkungen zur {T}heorie der beschr\"{a}nkten {B}ilinearformen mit
  unendlich vielen {V}er\"{a}nderlichen.
\newblock {\em J. Reine Angew. Math. 140\/} (1911), 1--28.

\bibitem{S}
{\sc Sinnamon, G.}
\newblock A note on the {S}tieltjes transformation.
\newblock {\em Proc. Roy. Soc. Edinburgh Sect. A 110}, 1-2 (1988), 73--78.

\bibitem{SW}
{\sc Soria, F., and Weiss, G.}
\newblock A remark on singular integrals and power weights.
\newblock {\em Indiana Univ. Math. J. 43}, 1 (1994), 187--204.

\bibitem{W}
{\sc Widder, D.~V.}
\newblock {\em The {L}aplace {T}ransform}.
\newblock Princeton Mathematical Series, v. 6. Princeton University Press,
  Princeton, N. J., 1946.
\newblock 2nd ed.

\end{thebibliography}

\end{document}